\documentclass[11pt]{article}
\usepackage[a4paper,margin=1in]{geometry}
\usepackage{amsmath,amsthm,amssymb,mathtools}
\usepackage{esint}
\usepackage{enumitem}
\usepackage{microtype}
\usepackage[colorlinks=true,linkcolor=blue,citecolor=blue,urlcolor=blue]{hyperref}
\usepackage{mathrsfs}
\numberwithin{equation}{section}
\allowdisplaybreaks[1]

\newtheorem{theorem}{Theorem}[section]
\newtheorem{proposition}[theorem]{Proposition}
\newtheorem{lemma}[theorem]{Lemma}

\theoremstyle{remark}
\newtheorem{remark}[theorem]{Remark}
\theoremstyle{definition}
\newtheorem{definition}[theorem]{Definition}

\newcommand{\D}{\mathbb D}
\newcommand{\Hh}{\mathcal H}
\newcommand{\A}{\mathcal A}

\newcommand{\T}{\mathscr T}
\newcommand{\Ssp}{\mathscr S}

\newcommand{\norm}[1]{\left\lVert #1\right\rVert}

\newcommand{\ip}[2]{\left\langle #1,#2\right\rangle}
\newcommand{\mueta}{\mu_\eta}

\title{The Bergman projection with operator weights}

\author{%
\textsc{Chunxu Xu}\textsuperscript{1}%
\thanks{E-mail: \href{mailto:1968385450@qq.com}{1968385450@qq.com}.}%
\thanks{Supported by the National Natural Science Foundation of China (Grant Nos.~12401154 and 12601234).}
\quad and \quad
\textsc{Jianxiang Dong}\textsuperscript{2}%
\thanks{E-mail: \href{mailto:jianxd@tsnu.edu.cn}{jianxd@tsnu.edu.cn}. Corresponding author.}%
\\[1.0ex]
\small\textsuperscript{1}School of Science, Nanjing Forestry University,\\[-0.15ex]
\small Nanjing 210037, P.\ R.\ China\\[0.45ex]
\small\textsuperscript{2}School of Mathematics and Statistics, Tianshui Normal University,\\[-0.15ex]
\small Tianshui 741000, P.\ R.\ China
}
\date{}

\makeatletter
\renewcommand{\@maketitle}{%
  \newpage
  \null
  \vskip -2.2em%
  \begin{center}%
    {\large\bfseries \@title\par}%
    \vskip 1.15em%
    {\normalsize \@author\par}%
  \end{center}%
  \vskip 0.9em%
}
\makeatother

\begin{document}
\maketitle

\begin{abstract}
We characterize boundedness of the Bergman projection on operator-weighted $L^p$ spaces over a separable Hilbert space, for $1<p<\infty$. Under the directional integrability conditions, the criterion is uniform boundedness of averages over Carleson squares in the unit disc. We obtain norm estimates independent of dimension. The proof uses a separated kernel expansion and a layer decomposition of the tent tree. No reverse H\"older assumption is needed. We also give a finer matrix estimate with a dimension-dependent constant. Boundedness is unchanged by compactly supported changes of the weight that preserve the integrability conditions.
\end{abstract}



\section{Introduction}
Weighted Bergman projections lead to a basic question: which conditions on a weight make the projection bounded? For scalar weights, the B\'ekoll\'e--Bonami theorem answers this question by averages over Carleson squares \cite{Bekolle82,BB78}. Aleman and Constantin obtained an operator-weighted version on $L^2$ over a separable Hilbert space \cite{AlemanConstantin12}. Their proof uses duality of analytic spaces and weighted estimates for derivatives.

We ask whether local averages also characterize boundedness for every $1<p<\infty$, with constants independent of dimension. Two difficulties arise. Finite-dimensional norm comparisons can introduce constants that grow with dimension. Also, general Bergman weights need not satisfy reverse H\"older inequalities \cite{AlemanPottReguera19}. We prove the characterization for every separable Hilbert space by a direct estimate on the tent tree. This estimate uses the averaging condition alone.

Let $\D$ be the unit disc and let $dA$ be normalized area measure. For $\eta>-1$, set
\[
 d\mueta(z)=(\eta+1)(1-|z|^2)^\eta dA(z).
\]
The Bergman projection is
\begin{equation}\label{eq:bergman}
 P_\eta f(z)=\int_\D\frac{f(\zeta)}{(1-\overline\zeta z)^{\eta+2}}\,d\mueta(\zeta).
\end{equation}
The kernel is scalar. Its power is taken with the analytic branch equal to $1$ at $z=0$. We use the same formula for Hilbert-space-valued functions.

Fix $1<p<\infty$ and put $p'=p/(p-1)$. All Hilbert spaces below are complex. Let $\Hh$ be a nonzero separable Hilbert space. Its inner product is linear in the first variable. Vector-valued measurability means strong measurability. Let $\mathcal B(\Hh)$ denote the bounded operators on $\Hh$. A weight $W$ has positive boundedly invertible values in $\mathcal B(\Hh)$ almost everywhere. The inverse need not be bounded uniformly in $z$. For each $e\in\Hh$, the map $z\mapsto W(z)e$ is measurable. The powers used below have the same property; this is justified in Section~\ref{sec:prelim-scalar}. Set $U=W^{1/p}$. We assume
\begin{equation}\label{eq:standing}
 \int_\D\|Ue\|^p\,d\mueta<\infty,
 \qquad
 \int_\D\|U^{-1}e\|^{p'}\,d\mueta<\infty
 \qquad(e\in\Hh).
\end{equation}
Define
\begin{equation}\label{eq:LpW}
 \|f\|_{L^p(W)}^p=\int_\D\|U(z)f(z)\|^p\,d\mueta(z).
\end{equation}
The space $L^p(W)$ consists of measurable $\Hh$-valued functions with finite norm. Functions equal almost everywhere are identified. These assumptions remain in force in the paper.

Normalize arc length so that $|\partial\D|=1$. For an arc $I\subset\partial\D$ with $0<|I|\le1$, put
\[
 S(I)=\{re^{it}:e^{it}\in I,\ 1-|I|<r<1\}.
\]
We also call $S(I)$ a tent. Its top half is the part where $1-|I|<|z|\le1-|I|/2$. We include $S(\partial\D)=\D$, up to a null set. Write $\fint_E=\mueta(E)^{-1}\int_E$ for a measurable set of positive measure. Let $\mathbf1_E$ denote the indicator of $E$. For a Carleson square $S$, define
\[
 E_Sf=\mathbf1_S\fint_S f\,d\mueta,
 \qquad
 [W]_{\mathbf B_p^{\rm av}(\eta)}
 =\sup_S\|E_S\|_{L^p(W)\to L^p(W)}^p.
\]
The averages are weak vector averages. Section~\ref{sec:prelim-scalar} gives their definition and a dense domain for \eqref{eq:bergman}.

We write $A\lesssim_\gamma B$ if $A\le C_\gamma B$. We write $A\asymp_\gamma B$ if both inequalities hold. The subscript lists the permitted dependence of the constant. Our main result is the following.
\begin{theorem}[Operator-weight characterization]\label{thm:operator-main}
Let $1<p<\infty$, $\eta>-1$, and let $W$ satisfy the standing assumptions on a nonzero separable Hilbert space. Then
\[
 P_\eta:L^p(W)\to L^p(W)\text{ is bounded}
 \quad\Longleftrightarrow\quad
 B:=[W]_{\mathbf B_p^{\rm av}(\eta)}<\infty.
\]
More precisely,
\[
 c_\eta B^{1/p}\le\|P_\eta\|_{L^p(W)\to L^p(W)}
 \le C_{p,\eta} B^{\gamma_p},
 \qquad \gamma_p=1+\frac1p+\frac1{p-1}.
\]
The constants do not depend on dimension. No directional reverse H\"older condition is assumed.
\end{theorem}

The main point is the estimate for general $p$ in infinite dimension. The theorem requires no integrability of the operator norms of $W$ and $W^{-1}$. Only the directional conditions in \eqref{eq:standing} are imposed. It also applies to noncommuting weights that cannot be compared with a single scalar weight; see Proposition~\ref{prop:noncommuting-example}.

The proof of necessity is the scalar-multiplier argument of \cite[Proposition~2.1]{AlemanConstantin12}. For sufficiency, we first write the kernel as a summable series of separated scalar factors on dyadic tents. We then split each tent into layers of top halves. The averaging bound gives geometric decay on these layers for both the weight and its dual. Fixed layers are disjoint. Summing the resulting estimates proves the upper bound. The expansion retains the complex kernel. Its coefficients have a summable bound that is uniform in the tent. This uniform bound allows us to pass from the tree estimate to the Bergman projection.

At $p=2$, almost orthogonality gives a smaller exponent.
\begin{theorem}\label{thm:hilbert-bound}
Let $p=2$ and let $W$ satisfy the standing assumptions on a separable Hilbert space. Put $B=[W]_{\mathbf B_2^{\rm av}(\eta)}$. If $B<\infty$, then
\begin{equation}\label{eq:hilbert-upper}
 \|P_\eta\|_{L^2(W)\to L^2(W)}\le C_\eta B^{3/2}.
\end{equation}
The constant is independent of the Hilbert-space dimension.
\end{theorem}

Here the local averaging identity is
\begin{equation}\label{eq:B2op-intro}
 \|E_S\|=
 \left\|\left(\fint_SW\,d\mueta\right)^{1/2}
             \left(\fint_SW^{-1}\,d\mueta\right)^{1/2}\right\|.
\end{equation}
Thus $B$ is the square of the norm-type characteristic. Theorem~\ref{thm:hilbert-bound} combines the separated kernel with the sparse method of Limani and Pott \cite{LimaniPott21}. Their sparse estimate is independent of dimension. Their Bergman application, \cite[Theorem~2.3]{LimaniPott21}, is stated for matrix weights. The separated expansion allows us to apply the sparse estimate on any separable Hilbert space. A precise comparison with the quantitative bound in \cite{AlemanConstantin12} is given in Remark~\ref{rem:AC-comparison}. We make no claim of sharpness for either upper bound.

For finite-dimensional weights, one can use a smaller power for general $p$. Define
\[
 [W]_{\mathbf B_p^{\rm R}(\eta)}
 =\sup_S\fint_S
 \left(\fint_S\|W(z)^{1/p}W(\zeta)^{-1/p}\|^{p'}\,d\mueta(\zeta)\right)^{p/p'}d\mueta(z).
\]
The norm inside the integral is the operator norm on $\mathbb C^d$.
\begin{theorem}\label{thm:matrix-main}
Let $\eta>-1$, $1<p<\infty$, and $d<\infty$.  Let $W:\D\to\mathbb C^{d\times d}$ be positive definite almost everywhere and satisfy the standing assumptions stated above.  Then the following are equivalent:
\begin{enumerate}[label=(\roman*)]
\item $P_\eta$ extends boundedly on $L^p(W)$;
\item $[W]_{\mathbf B_p^{\rm av}(\eta)}<\infty$;
\item $[W]_{\mathbf B_p^{\rm R}(\eta)}<\infty$.
\end{enumerate}
Moreover,
\begin{equation}\label{eq:main-lower}
 [W]_{\mathbf B_p^{\rm av}(\eta)}^{1/p}
 \lesssim_\eta\|P_\eta\|_{L^p(W)\to L^p(W)},
\end{equation}
and
\begin{equation}\label{eq:main-upper}
 \|P_\eta\|_{L^p(W)\to L^p(W)}
 \lesssim_{p,d,\eta}
 [W]_{\mathbf B_p^{\rm R}(\eta)}^{\theta_p},
 \qquad
 \theta_p=1+\frac1{p-1}-\frac1p.
\end{equation}
The two characteristics are comparable with constants depending on $p,d$.
\end{theorem}

The quantitative estimate in Theorem~\ref{thm:matrix-main} follows from the nonhomogeneous matrix sparse theory of Benito-de la Cigo\~na, Borges, D'Emilio, Pasquariello, and Wagner \cite{BenitoEtAl26}. We give its tent-tree proof in Appendix~\ref{sec:matrix-class}. Earlier matrix Bergman estimates at $p=2$ are due to Huo and Wick \cite{HuoWick24}. Their work and \cite{HuoWagnerWick21} also treat more general domains.

Matrix-weighted Fock projections for general $p$ are studied in \cite{ChenWangFock25}. Bergman operators in noncommutative $L^p$ spaces are considered in \cite{Tian26}. The latter spaces use a different norm from \eqref{eq:LpW}.

Hyt\"onen, Li, Yang, and Yuan prove an operator-weighted sparse estimate on Euclidean cubes \cite[Corollary~8.4]{HLYY26}. Their proof uses a reverse H\"older inequality obtained from the $A_p$ condition on all subcubes. Our condition is imposed on boundary tents. These two families of sets are different, so the Euclidean result does not apply directly. Each tent leaves a top half of positive relative measure. The averaging bound forces geometric decay of directional weight mass along the descendants. Our layer estimate uses this decay in place of a reverse H\"older inequality.

We also prove stability under changes of the weight on a compact subset. We use it to give a noncommuting example with no global reverse H\"older gain. Finally, we give finite-dimensional compressions that preserve the pointwise weighted norm. These results are Propositions~\ref{prop:compact-stability}--\ref{prop:finite-sections}. The main estimates are not claimed to be sharp. Improving their powers remains a separate question.

Section~\ref{sec:prelim-scalar} fixes the dual-space notation and compares the normalizations. Section~\ref{sec:necessity-matrix} proves necessity. Section~\ref{sec:separated} proves sufficiency. Section~\ref{sec:examples} gives examples, stability under changes on compact subsets, and exact finite-dimensional compressions. The appendix contains the finer matrix estimate.

\section{Weights and local averages}\label{sec:prelim-scalar}
We first justify the measurability used in the definitions. For a positive integer $m$, let
\[
 \Omega_m=\{z\in\D:m^{-1}I\le W(z)\le mI\}.
\]
The order means that the corresponding quadratic forms are ordered. These sets are measurable: it suffices to test the two inequalities on a countable dense subset of the unit sphere. They increase to a set of full measure. If simple functions $f_n$ converge pointwise to $f$, then $W(z)f_n(z)\to W(z)f(z)$, since each $W(z)$ is bounded. Thus multiplication by $W$ preserves measurability. In particular, polynomial powers of $W$ applied to a fixed vector are measurable. On $\Omega_m$, approximate $t^r$ uniformly by polynomials on $[m^{-1},m]$, for $r\in\mathbb R$. Functional calculus then shows that $W^re$ is measurable for every $e$. The same simple-function argument proves measurability of $W^rf$.

Under the pairing $\int_\D\langle f,g\rangle\,d\mueta$, the dual of $L^p(W)$ has norm
\begin{equation}\label{eq:dualweight}
 \norm{g}_{p',W^\#}^{p'}
 =\int_\D \norm{W(z)^{-1/p}g(z)}^{p'}\,d\mueta(z),
 \qquad p'=\frac p{p-1}.
\end{equation}
Here $W^\#=W^{-p'/p}$, so $(W^\#)^{1/p'}=W^{-1/p}$. At $p=2$, this gives $W^\#=W^{-1}$. To prove the duality statement, use the onto isometry $f\mapsto Uf$ and its inverse $F\mapsto U^{-1}F$. Ordinary Hilbert-space-valued $L^p$ duality represents each bounded linear functional as
\[
 f\longmapsto\int_\D\langle Uf,G\rangle\,d\mueta,
 \qquad G\in L^{p'}(\mueta;\Hh).
\]
Put $g=UG$. Then $U^{-1}g=G$ and $\langle Uf,G\rangle=\langle f,g\rangle$. The norm of this functional is exactly the norm in \eqref{eq:dualweight}.

Fix a Carleson square $Q$. The maps
\[
 e\longmapsto \mathbf1_Q Ue\in L^p(Q;\Hh),
 \qquad
 e\longmapsto \mathbf1_Q U^{-1}e\in L^{p'}(Q;\Hh)
\]
are defined on all of $\Hh$ by \eqref{eq:standing}. They have closed graphs. For example, suppose $e_n\to e$ in $\Hh$ and $Ue_n\to F$ in $L^p(Q;\Hh)$. A subsequence converges almost everywhere to $F$. But boundedness of each $U(z)$ gives $U(z)e_n\to U(z)e$ pointwise. Thus $F=Ue$. The proof for $U^{-1}$ is the same. The closed graph theorem gives constants $C_Q,C_Q^*<\infty$ such that
\begin{equation}\label{eq:directional-uniform-local}
 \left(\fint_Q\|U(z)e\|^p\,d\mueta(z)\right)^{1/p}\le C_Q\|e\|,
 \qquad
 \left(\fint_Q\|U(z)^{-1}e\|^{p'}\,d\mueta(z)\right)^{1/p'}\le C_Q^*\|e\|.
\end{equation}
Every $f\in L^p(W)$ has a weak vector average on $Q$. Weighted H\"older's inequality gives
\[
 \left|\int_Q\langle f(z),h\rangle\,d\mueta(z)\right|
 \le\mueta(Q)^{1/p'}C_Q^*\|\mathbf1_Qf\|_{L^p(W)}\|h\|.
\]
Thus the integral defines a bounded conjugate-linear functional of $h$. It is represented by a unique vector of $\Hh$. Dividing this vector by $\mueta(Q)$ defines $\fint_Q f\,d\mueta$. This definition does not require $\int_Q\|f\|\,d\mueta<\infty$. The same argument applies to $bf$ for bounded scalar $b$.

If $v=\fint_Q f\,d\mueta$, the preceding estimate gives
\[
 \|v\|\le\mueta(Q)^{-1/p}C_Q^*\|\mathbf1_Qf\|_{L^p(W)}.
\]
The first inequality in \eqref{eq:directional-uniform-local} then gives
$\|E_Qf\|_{L^p(W)}\le C_QC_Q^*\|\mathbf1_Qf\|_{L^p(W)}$.
In particular, each average is bounded. Also $E_Q^2=E_Q$, and $E_Q$ fixes every function $\mathbf1_Qe$. Taking $e\ne0$ shows that $\|E_Q\|\ge1$.

Define the initial domain of the Bergman kernel by
\[
 \mathscr D_W
 :=\{U^{-1}F:\ F\text{ is a bounded, compactly supported, finite-valued simple }\Hh\text{-function}\}.
\]
Multiplication by $U$ maps $L^p(W)$ isometrically onto $L^p(\mueta;\Hh)$. Simple functions with compact support are dense there. Thus $\mathscr D_W$ is dense in $L^p(W)$. To check Bochner integrability, write $F=\sum_{j=1}^N\mathbf1_{E_j}e_j$ on disjoint measurable sets. H\"older's inequality gives
\[
 \int_\D\|U^{-1}F\|\,d\mueta
 \le\sum_{j=1}^N\mueta(E_j)^{1/p}
       \left(\int_{E_j}\|U^{-1}e_j\|^{p'}\,d\mueta\right)^{1/p'}<\infty.
\]
For fixed $z\in\D$, the scalar kernel is bounded in $\zeta$. Hence \eqref{eq:bergman} is a Bochner integral on $\mathscr D_W$. Boundedness of $P_\eta$ on $L^p(W)$ means a bounded extension of this map from $\mathscr D_W$. Such an extension is unique by density.

At $p=2$, the bounded positive form
\[
 (e,h)\longmapsto\fint_Q
 \langle W(z)^{1/2}e,W(z)^{1/2}h\rangle\,d\mueta(z)
\]
defines the weak average $\fint_QW\,d\mueta$. The same construction applies to $W^{-1}$. Both averages are invertible. Indeed, Cauchy--Schwarz gives
\[
 \|e\|^2=\fint_Q\langle W^{1/2}e,W^{-1/2}e\rangle\,d\mueta
 \le\left(\fint_Q\|W^{1/2}e\|^2\,d\mueta\right)^{1/2}C_Q^*\|e\|.
\]
Hence $\fint_QW\,d\mueta\ge(C_Q^*)^{-2}I$. Interchanging $W$ and $W^{-1}$ gives
$\fint_QW^{-1}\,d\mueta\ge C_Q^{-2}I$. In particular, the hypotheses of \cite[Theorem~1.1]{AlemanConstantin12} hold at $p=2$.

\begin{remark}\label{rem:AC-comparison}
Let $Q_2=[W]_{\mathbf B_2^{\rm av}(\eta)}^{1/2}$ on the tent basis used here. Let $Q_2^{\rm AC}$ denote the characteristic in \cite[Eq.~(1.1)]{AlemanConstantin12}. Their windows are
\[
 S_{\rm AC}(\theta,h)=\{re^{it}:1-h<r<1,\ |t-\theta|<h\},\qquad 0<h<1.
\]
Each lies in our square with height $h$ and the same center. The measure ratio is $\pi$. For a measurable set $E$ of positive measure, let $E_Ef=\mathbf1_E\fint_Ef\,d\mueta$. If $E\subset F$, then
\[
 E_E=\frac{\mueta(F)}{\mueta(E)}M_EE_FM_E,
 \qquad M_Ef=\mathbf1_Ef.
\]
Consequently $Q_2^{\rm AC}\le\pi Q_2$. Their fifth-power upper bound therefore gives $\|P_\eta\|\lesssim_\eta Q_2^5$ on our basis. Theorem~\ref{thm:hilbert-bound} gives $\|P_\eta\|\lesssim_\eta Q_2^3$. This compares two bounds in the same characteristic $Q_2$. It does not identify $Q_2^{\rm AC}$ with $Q_2$.
\end{remark}

For a positive scalar weight $w$, set
\[
 [w]_{B_p(\eta)}=\sup_S\left(\fint_Sw\,d\mueta\right)
 \left(\fint_Sw^{-1/(p-1)}\,d\mueta\right)^{p-1}.
\]
The scalar theorem \cite{Bekolle82,BB78,PottReguera13} states that $P_\eta$ is bounded on $L^p(w)$ if and only if this number is finite. The operator with the absolute kernel is
\[
 P_\eta^+f(z)=\int_\D\frac{f(\zeta)}{|1-\overline\zeta z|^{\eta+2}}\,d\mueta(\zeta).
\]
It is positive: it sends nonnegative scalar functions to nonnegative functions. It satisfies
\begin{equation}\label{eq:sharp-scalar}
 \|P_\eta^+\|_{L^p(w)\to L^p(w)}
 \le C_{p,\eta}[w]_{B_p(\eta)}^{\max\{1,1/(p-1)\}}.
\end{equation}

\begin{proposition}\label{thm:scalar-vector}
For any nonzero Hilbert space, $P_\eta$ is bounded on $L^p(w\,d\mueta;\Hh)$ if and only if $w\in B_p(\eta)$. The bound in \eqref{eq:sharp-scalar} also holds for this vector-valued operator.
\end{proposition}
\begin{proof}
Suppose $w\in B_p(\eta)$. For a bounded, compactly supported simple function $f$, the triangle inequality gives
\[
 \|P_\eta f(z)\|\le P_\eta^+(\|f\|)(z).
\]
Taking the scalar weighted $L^p$ norm and applying \eqref{eq:sharp-scalar} proves the asserted estimate on this class. The class is dense in $L^p(w\,d\mueta;\Hh)$, so the operator extends boundedly. Since $w^{-1/(p-1)}$ is integrable, weighted H\"older's inequality shows that convergence in this space implies convergence in $L^1$. Thus the extension agrees with the kernel integral on $\mathscr D_{wI}$. For necessity, take $f=\varphi e$ with a fixed unit vector $e$. Both the function norm and the projection norm are then the corresponding scalar norms. The scalar theorem gives $w\in B_p(\eta)$.
\end{proof}

If $cwI\le W\le CwI$, the spectral theorem gives
\[
 c^{1/p}\|f\|_{L^p(w;\Hh)}\le\|f\|_{L^p(W)}
 \le C^{1/p}\|f\|_{L^p(w;\Hh)}.
\]
Hence the same characterization holds for such $W$. In particular,
\[
 \|P_\eta\|_{L^p(W)\to L^p(W)}
 \le (C/c)^{1/p}C_{p,\eta}
       [w]_{B_p(\eta)}^{\max\{1,1/(p-1)\}}.
\]

\begin{proposition}\label{prop:power}
For $w_\beta(z)=(1-|z|^2)^\beta$,
\[
 w_\beta\in B_p(\eta)
 \quad\Longleftrightarrow\quad
 -(\eta+1)<\beta<(p-1)(\eta+1).
\]
Within this range,
\[
 [w_\beta]_{B_p(\eta)}
 =\frac{\eta+1}{\eta+1+\beta}
   \left(\frac{\eta+1}{\eta+1-\beta/(p-1)}\right)^{p-1}.
\]
\end{proposition}
\begin{proof}
Let $h=|I|$ and put $q_h=2h-h^2$. In polar coordinates, the angular integral gives a factor $h$. The substitution $t=1-r^2$ then gives
\[
 \int_{S(I)}w_\beta\,d\mueta
 =h(\eta+1)\int_0^{q_h}t^{\eta+\beta}\,dt.
\]
This integral is finite exactly when $\eta+\beta>-1$. In that case,
\[
 \mueta(S(I))=h q_h^{\eta+1},\qquad
 \fint_{S(I)}w_\beta\,d\mueta
 =\frac{\eta+1}{\eta+1+\beta}q_h^\beta,
\]
where the measure formula is the case $\beta=0$. Apply the same computation with $-\beta/(p-1)$ in place of $\beta$. Finiteness of both integrals is equivalent to the stated strict inequalities. Multiplying the two averages cancels the powers of $q_h$. The product is independent of $I$, which gives the exact characteristic. At either endpoint, one of the integrals contains $\int_0^{q_h}t^{-1}\,dt$ and diverges.
\end{proof}

For instance, at $\eta=0$, the weight $(1-|z|^2)^{3/4}$ belongs to $B_2$ but not to $B_{3/2}$. The weight $(1-|z|^2)^{3/2}$ belongs to $B_3$ but not to $B_2$. Thus the $L^2$ condition cannot be retained unchanged for other exponents. The dual weight in \eqref{eq:dualweight} also changes with $p$.

\section{The necessity part}\label{sec:necessity-matrix}
The argument applies to any separable Hilbert space. We use the localization identity from \cite[Proposition~2.1]{AlemanConstantin12}.
Let
\[
 k_{\lambda}^{\eta}(z)=
 \frac{(1-|\lambda|^2)^{(\eta+2)/2}}{(1-\overline\lambda z)^{\eta+2}}
\]
be the normalized scalar reproducing kernel, so $\|k_\lambda^\eta\|_{L^2(\mueta)}=1$. Define the normalized kernel projection
\[
 P_{\eta,\lambda}f
 =k_{\lambda}^{\eta}
 \int_\D f(\zeta)\overline{k_{\lambda}^{\eta}(\zeta)}\,d\mueta(\zeta).
\]
For general $f\in L^p(W)$, this vector integral is understood weakly. For fixed $\lambda$, the function $k_\lambda^\eta$ is bounded, so Section~\ref{sec:prelim-scalar} gives its existence.

\begin{lemma}\label{lem:Plambda}
Assume $1<p<\infty$ and $P_\eta$ is bounded on $L^p(W)$.  Then
\[
 \sup_{\lambda\in\D}
 \|P_{\eta,\lambda}\|_{L^p(W)\to L^p(W)}
 \le C_\eta\|P_\eta\|_{L^p(W)\to L^p(W)}.
\]
The constant is independent of $p$, $W$, and the Hilbert-space dimension.
\end{lemma}

\begin{proof}
Put $\alpha=\eta+2$ and
$\varphi_\lambda(z)=(\lambda-z)/(1-\overline\lambda z)$.
The identity
\[
 1-\varphi_\lambda(z)\overline{\varphi_\lambda(\zeta)}
 =\frac{(1-|\lambda|^2)(1-\overline\zeta z)}
 {(1-\overline\lambda z)(1-\lambda\overline\zeta)}
\]
gives
\[
 k_\lambda^\eta(z)\overline{k_\lambda^\eta(\zeta)}
 =K_\eta(z,\zeta)
   (1-\varphi_\lambda(z)\overline{\varphi_\lambda(\zeta)})^\alpha,
 \qquad K_\eta(z,\zeta)=(1-\overline\zeta z)^{-\alpha}.
\]
All powers use the analytic branch equal to $1$ at the origin for $(1-t)^\alpha$. The algebraic identity holds at power $1$. The factors have analytic logarithms on the bidisc in $z$ and $\overline\zeta$. The difference of the logarithms is a continuous multiple of $2\pi i$, hence is constant. It vanishes at $z=\zeta=\lambda$. This proves the identity for real $\alpha$.
Write
\[
 (1-t)^\alpha=\sum_{n\ge0}d_nt^n,
 \qquad d_n=(-1)^n\binom\alpha n.
\]
If $\alpha$ is an integer, this sum is finite. Otherwise, for $n>\alpha$,
\[
 |d_{n+1}|=\left(1-\frac{\alpha+1}{n+1}\right)|d_n|.
\]
Iteration and $\log(1-u)\le-u$ give $|d_n|\le C_\alpha(n+1)^{-\alpha-1}$. Thus in both cases $D_\alpha:=\sum_n|d_n|<\infty$.
Consider the operator series
\[
 S_\lambda
 :=\sum_{n\ge0}d_n M_{\varphi_\lambda^n}P_\eta
                       M_{\overline{\varphi_\lambda}^{\,n}}.
\]
Here $M_a f=af$. Since $|\varphi_\lambda|\le1$, both multipliers are contractions on $L^p(W)$. The series converges in operator norm and $\|S_\lambda\|\le D_\alpha\|P_\eta\|$. We now show that $S_\lambda=P_{\eta,\lambda}$.
We first justify the kernel formula after scalar multiplication. Let $f\in\mathscr D_W$ and let $b$ be bounded and measurable. Choose finite-valued scalar functions $b_m$ with $\|b_m-b\|_\infty\to0$. Then $b_mf\in\mathscr D_W$ and
\[
 \|(b_m-b)f\|_{L^p(W)}\le\|b_m-b\|_\infty\|f\|_{L^p(W)},
 \qquad
 \|(b_m-b)f\|_{L^1}\le\|b_m-b\|_\infty\|f\|_{L^1}.
\]
Thus the bounded extension of $P_\eta$ agrees with its kernel integral on $bf$.
Indeed, the first convergence gives convergence of the operator values in $L^p(W)$. The second gives pointwise convergence of the kernel integrals, since $|K_\eta(z,\zeta)|\le(1-|z|)^{-\alpha}$ for fixed $z$.
An almost everywhere convergent subsequence after multiplication by $U$ identifies these limits; apply $U(z)^{-1}$ pointwise.

Apply this observation with $b=\overline{\varphi_\lambda}^{\,n}$.
For fixed $z$, absolute summability of $(d_n)$ permits termwise integration against the Bochner integrable function $f$.
The scalar identity gives the kernel defining $P_{\eta,\lambda}$.
The same subsequence argument identifies this pointwise limit with the operator-norm sum on $\mathscr D_W$.
For completeness, \eqref{eq:directional-uniform-local} with $Q=\D$ and $\mueta(\D)=1$ gives
\[
 \left\|\int_\D f\,\overline{k_\lambda^\eta}\,d\mueta\right\|
 \le C_\D^*\|k_\lambda^\eta\|_\infty\|f\|_{L^p(W)},
 \qquad
 \|k_\lambda^\eta e\|_{L^p(W)}
 \le C_\D\|k_\lambda^\eta\|_\infty\|e\|.
\]
Thus $P_{\eta,\lambda}$ is bounded for each fixed $\lambda$. Density identifies it with $S_\lambda$. The uniform estimate comes from the bound for $S_\lambda$, whose constant is $D_\alpha$ and is independent of $\lambda$, $p$, and $W$.
This is the localization argument of \cite{AlemanConstantin12} in a form valid for every $p$.
\end{proof}

\begin{lemma}\label{lem:kernel-square}
For every Carleson square $S=S(I)$ there exists $\lambda_S\in\D$, determined by the midpoint and length of $I$, such that
\begin{equation}\label{eq:kernel-square}
 |k_{\lambda_S}^{\eta}(z)|
 \asymp_\eta \mueta(S)^{-1/2},
 \qquad z\in S.
\end{equation}
\end{lemma}

\begin{proof}
Let $h=|I|$ and let $e^{i\theta_I}$ be the midpoint of $I$. Set $\lambda_S=(1-h)e^{i\theta_I}$. If $z=re^{it}\in S$, choose the angular difference with $|t-\theta_I|\le\pi h$. Then
\[
 h\le |1-\overline{\lambda_S}z|
 \le h+(1-r)+|t-\theta_I|\le(2+\pi)h,
 \qquad h\le1-|\lambda_S|^2=2h-h^2\le2h.
\]
These bounds also hold for the root, where $h=1$ and $\lambda_S=0$. Since $\mueta(S)\asymp_\eta h^{\eta+2}$, substitution proves the claim.
\end{proof}

\begin{proposition}\label{prop:matrix-necessity}
If $P_\eta$ is bounded on $L^p(W)$, then
\begin{equation}\label{eq:necessity-bound}
 [W]_{\mathbf B_p^{\rm av}(\eta)}^{1/p}
 \lesssim_\eta \|P_\eta\|_{L^p(W)\to L^p(W)}.
\end{equation}
The constant is independent of dimension.
\end{proposition}

\begin{proof}
Fix $S$ and choose $\lambda_S$ as in Lemma~\ref{lem:kernel-square}.  Set
\[
 m_S(z)=\mueta(S)^{1/2}k_{\lambda_S}^{\eta}(z).
\]
On functions supported in $S$ there is the exact factorization
\[
 \mathbf1_SP_{\eta,\lambda_S}\mathbf1_S
 =M_{m_S}E_SM_{\overline{m_S}}.
\]
By \eqref{eq:kernel-square}, both $M_{m_S}$ and its inverse on $S$ are bounded scalar multipliers with norms controlled only by $\eta$.  More explicitly, extend $m_S^{-1}$ and $\overline{m_S}^{-1}$ by zero off $S$. Then
\[
 E_S=M_{\mathbf1_S/m_S}\,P_{\eta,\lambda_S}\,
       M_{\mathbf1_S/\overline{m_S}}.
\]
Both scalar multipliers in this formula are bounded. Therefore
\[
 \|E_S\|_{L^p(W)\to L^p(W)}
 \lesssim_\eta
 \|P_{\eta,\lambda_S}\|_{L^p(W)\to L^p(W)}.
\]
Taking the supremum over $S$ and applying Lemma~\ref{lem:Plambda} gives \eqref{eq:necessity-bound}.  The argument used only scalar multiplication and therefore does not depend on finite dimensionality.
\end{proof}

\section{The sufficiency part}\label{sec:separated}
We factor the localized scalar kernel into separated variables. The coefficients have a summable bound independent of the tent scale. The weight enters only after this step, through Lemma~\ref{lem:layer-tree}.

\subsection{A separated kernel expansion}

Put
\[
 \alpha=\eta+2>1,
 \qquad K_\eta(z,\zeta)=(1-\overline\zeta z)^{-\alpha}.
\]
Bisect the circle into arcs and continue by bisecting each arc. Together with the full circle, these arcs form a dyadic grid. Take arcs half-open so that each generation is a partition. Endpoints have measure zero and do not affect the estimates. Write $\T$ for the full family of their Carleson tents. Any two tents are disjoint or one contains the other.

We use one such grid and its rotations by one third and two thirds of a turn. Every arc $J$ lies in an arc $I$ from one of these three grids with $|I|\le6|J|$. If $|J|\ge1/6$, take the full circle. Otherwise choose a dyadic length $H$ with $3|J|<H\le6|J|$; then $H\le1/2$. At this scale, the combined endpoints of the three grids are spaced by $H/3$. The arc $J$ meets an endpoint of at most one grid. One of the other grids therefore has an arc $I$ of length $H$ containing $J$.

A family $\mathscr S$ is $\delta$-sparse, where $0<\delta<1$, if there are pairwise disjoint measurable sets $G_Q\subset Q$ with $\mueta(G_Q)\ge\delta\mueta(Q)$ for every $Q\in\mathscr S$.

\begin{lemma}\label{lem:full-tree-sparse}
Let $Q=S(I)$ and let $Q_1,Q_2$ be the two tents associated with the dyadic children of $I$.  Then there is a number $\kappa_\eta<1$, depending only on $\eta$, such that
\[
 \mueta(Q_1\cup Q_2)\le \kappa_\eta\mueta(Q).
\]
Consequently every full dyadic tent tree is sparse.
\end{lemma}

\begin{proof}
Write $h=|I|$.  A direct polar-coordinate integration gives
\[
 \mueta(S(I))=h^\alpha(2-h)^{\alpha-1},
\]
under our normalization of arc length and area measure.
Hence
\[
 \frac{\mueta(Q_1)+\mueta(Q_2)}{\mueta(Q)}
 =2^{1-\alpha}
 \left(\frac{2-h/2}{2-h}\right)^{\alpha-1}
 \le \left(\frac34\right)^{\alpha-1}=:\kappa_\eta<1.
\]
The inequality uses $(2-h/2)/(2-h)\le3/2$ for $0<h\le1$ and $\alpha-1>0$. It includes the root $h=1$.
Take $G_Q=Q\setminus(Q_1\cup Q_2)$. If $R$ is a proper descendant of $Q$, then $R$ lies in one of the two children of $Q$, so $R\cap G_Q=\varnothing$. Tents in different branches are disjoint. Hence the sets $G_Q$ are pairwise disjoint. They satisfy $\mueta(G_Q)\ge(1-\kappa_\eta)\mueta(Q)$ and are exactly the top halves defined in the introduction.
\end{proof}

\begin{theorem}[Uniform separated tent factorization]\label{thm:separated-factorization}
There are three dyadic tent trees $\T_1,\T_2,\T_3$ with the following property. Put
$\mathcal N=\mathbb Z^2\times\mathbb Z^2$. There are numbers $C_\nu\ge0$ such that
\begin{equation}\label{eq:l1coeff}
 \sum_{\nu\in\mathcal N}C_\nu<\infty.
\end{equation}
For every $j$, $Q\in\T_j$, and $\nu\in\mathcal N$, there are measurable scalar functions
$a_{Q,\nu}^{(j)}$, $b_{Q,\nu}^{(j)}$ that vanish off $Q$, and scalars $c_{Q,\nu}^{(j)}$, such that
\[
 |a_{Q,\nu}^{(j)}|\le1,\qquad
 |b_{Q,\nu}^{(j)}|\le1,\qquad
 |c_{Q,\nu}^{(j)}|\le C_\nu,
\]
and
\begin{equation}\label{eq:separated-kernel}
 K_\eta(z,\zeta)
 =\sum_{j=1}^3\sum_{Q\in\T_j}\frac1{\mueta(Q)}
   \sum_{\nu\in\mathcal N}c_{Q,\nu}^{(j)}
       a_{Q,\nu}^{(j)}(z)b_{Q,\nu}^{(j)}(\zeta)
\end{equation}
for almost every $(z,\zeta)\in\D\times\D$.
For each fixed pair, only finitely many tents contribute. For each tent, the inner series converges absolutely and uniformly. More precisely, for $\nu=(r,q)\in\mathbb Z^2\times\mathbb Z^2$, we may take
\[
 C_{(r,q)}=C_\eta(1+|r|+|q|)^{-6}.
\]
The constant depends only on $\eta$, after the cutoffs and the three grids have been fixed as below.
\end{theorem}

\begin{proof}
We construct the pieces explicitly. The main point is a bound for six derivatives that is uniform over all tents. This gives one summable bound for all Fourier coefficients.

\smallskip\noindent\emph{1. Divide the kernel by size.}
Choose a smooth nonincreasing function $\varrho:[0,\infty)\to[0,1]$ such that
$\varrho=1$ on $[0,1]$ and $\varrho=0$ on $[2,\infty)$. Set
\[
 \psi(t)=\varrho(t)-\varrho(2t),\qquad
 k_0=6,\qquad L=32,\qquad
 \psi_{\rm root}(t)=1-\varrho(2^{k_0}t).
\]
Then $\psi\ge0$ and $\psi$ vanishes off $[1/2,2]$. A telescoping sum gives
\begin{equation}\label{eq:delta-partition}
 1=\psi_{\rm root}(t)+\sum_{k\ge k_0}\psi(2^kt)
 \qquad(t>0).
\end{equation}
Only a bounded number of terms in this sum are nonzero at a fixed $t$.

Write $\delta(z,\zeta)=|1-\overline\zeta z|$. If the term at scale $h=2^{-k}$ is nonzero, then
\begin{equation}\label{eq:delta-scale}
 \tfrac12h\le\delta(z,\zeta)\le2h,\qquad
 1-|z|\le2h,\qquad 1-|\zeta|\le2h.
\end{equation}
The last two inequalities follow from $1-|z||\zeta|\le\delta(z,\zeta)$.
Since $k\ge6$, both moduli exceed $1/2$.
Write $z=re^{2\pi ix}$ and $\zeta=\rho e^{2\pi iy}$, where $x,y\in\mathbb R/\mathbb Z$.
Let $d(x,y)=\min_{n\in\mathbb Z}|x-y-n|$ be circular distance.
For an angular difference $\theta\in[-\pi,\pi]$,
\[
 \delta^2=(1-r\rho)^2+2r\rho(1-\cos\theta)
 \ge\theta^2/\pi^2.
\]
Here we used $r\rho\ge1/4$ and $1-\cos\theta\ge2\theta^2/\pi^2$.
Thus
\begin{equation}\label{eq:depth-h}
 d(x,y)=|\theta|/(2\pi)\le\delta/2\le h.
\end{equation}

Choose a nonnegative smooth function $\chi$ supported in $(-1,1)$ such that
$\sum_{m\in\mathbb Z}\chi(t-m)=1$ for all $t\in\mathbb R$.
For example, divide a smooth bump, positive on $[-1/2,1/2]$, by the sum of its integer translates.
For $0\le m<2^k$, define the periodic function
\[
 \chi_{k,m}(x)=\sum_{\ell\in\mathbb Z}
       \chi(2^kx-m-\ell2^k).
\]
These functions sum to one on the circle. Each is supported where $d(x,mh)<h$.
At a fixed $x$, at most three are nonzero. For every integer $a\ge0$,
\[
 |\chi_{k,m}^{(a)}(x)|\le C_a h^{-a}.
\]
Now put
\[
 K_{k,m,n}(z,\zeta)
 =K_\eta(z,\zeta)\psi(2^k\delta(z,\zeta))
      \chi_{k,m}(x)\chi_{k,n}(y).
\]
Set this expression equal to zero if either point is the origin. The scale cutoff already vanishes when either modulus is at most $1/2$.
Using \eqref{eq:delta-partition} and the angular partition, we obtain
\begin{equation}\label{eq:smooth-decomp}
 K_\eta=K_\eta\psi_{\rm root}(\delta)
       +\sum_{k\ge k_0}\sum_{m,n=0}^{2^k-1}K_{k,m,n}.
\end{equation}

\smallskip\noindent\emph{2. Place each piece in one tent.}
Use the three grids obtained by rotations through $0,1/3,2/3$ of a fixed dyadic grid.
For a nonzero piece $K_{k,m,n}$, \eqref{eq:depth-h} gives
$d(mh,nh)\le3h$.
The two angular supports therefore lie in an arc of length at most $5h$.
Enlarge this arc by $h$ at each end. Call the resulting arc $J$; its length is at most $7h$.

Choose the dyadic length
\[
 H=Lh=32h=2^{-(k-5)}\le1/2.
\]
At this length, the boundary points of the three grids are spaced by $H/3$.
Since $|J|\le7h<H/3$, $J$ meets a boundary point of at most one grid.
In one of the other grids, it lies inside a single arc $I$ of length $H$.
Choose the first such grid in a fixed order. This assigns the piece to exactly one tent $Q=S(I)$.
The angular supports lie inside $I$. The depth bounds in \eqref{eq:delta-scale} are less than $H$.
Consequently the piece vanishes off $Q\times Q$.

Fix a grid and a nonroot tent $Q$ of height $H$. A piece assigned to $Q$ must have $h=H/L$.
There is therefore only one possible scale $k$.
Both angular centers $mh,nh$ belong to $I$. There are at most $L+2$ such centers in $I$.
Thus at most
\[
 M_0=(L+2)^2
\]
pieces can be assigned to $Q$.
Denote their pairs $(m,n)$ by $\mathcal P_{j,Q}$ and add these pieces:
\[
 K_Q^{(j)}=\sum_{(m,n)\in\mathcal P_{j,Q}}K_{k,m,n}.
\]
An empty sum is zero. Assign the root term in \eqref{eq:smooth-decomp} to the root of the first tree.
Set the root pieces in the other two trees equal to zero.
Each piece has now been assigned once, so
\[
 K_\eta=\sum_{j=1}^3\sum_{Q\in\T_j}K_Q^{(j)}.
\]
The exact measure formula in Lemma~\ref{lem:full-tree-sparse} gives
\begin{equation}\label{eq:hQ}
 \mueta(Q)=H^\alpha(2-H)^{\alpha-1},\qquad
 A_Q:=\mueta(Q)H^{-\alpha}=(2-H)^{\alpha-1}.
\end{equation}
In particular, $A_Q$ is bounded independently of $Q$.

\smallskip\noindent\emph{3. Use the same coordinates on every nonroot tent.}
Let $x_Q$ be the normalized angular midpoint of $I$. Write
\[
 z=(1-Hs)e^{2\pi i(x_Q+Ht)},\qquad
 \zeta=(1-Hs')e^{2\pi i(x_Q+Ht')},
\]
where $(s,t),(s',t')$ belong to the fixed rectangle
$\mathcal R=[0,1]\times[-1/2,1/2]$. We write $\mathcal R^2=\mathcal R\times\mathcal R$ for the set of pairs of these coordinates.
Then $1-\overline\zeta z=H D_H$, where
\[
 D_H(s,t,s',t')
 =\frac{1-(1-Hs)(1-Hs')e^{iHv}}{H},
 \qquad v=2\pi(t-t').
\]
Expanding the numerator gives
\[
 D_H=(s+s'-Hss')e^{iHv}+\frac{1-e^{iHv}}{H}.
\]
Since
\[
 e^{iHv}-1=iHv\int_0^1e^{iHvu}\,du,
\]
we obtain
\[
 D_H=(s+s'-Hss')e^{iHv}
       -iv\int_0^1e^{iHvu}\,du.
\]
The right side is smooth in all five real variables, including $H=0$.
At $H=0$, its value is $s+s'-iv$.
Compactness of $[0,1/2]\times[-2,2]^4$ therefore gives, for each
integer $N\ge0$,
\[
 \sup_{0<H\le1/2}\;
 \max_{|\gamma|\le N}\;
 \sup_{(s,t,s',t')\in[-2,2]^4}
 |\partial^\gamma D_H(s,t,s',t')|
 \le C_N.
\]
Here $\partial^\gamma$ denotes a partial derivative of total order
$|\gamma|$ in the four coordinate variables.
The constant $C_N$ is independent of $H$.
On $\mathcal R^2$, put $a=(1-Hs)(1-Hs')$. Then $0\le a\le1$, so
\[
 \operatorname{Re}D_H
 =\frac{1-a\cos(Hv)}{H}
 \ge\frac{1-a}{H}\ge0.
\]
This last inequality is asserted only on $\mathcal R^2$.

Let $F_Q^{(j)}$ be $\mueta(Q)K_Q^{(j)}$ in these coordinates.
This is a pointwise expression for the kernel.
No integration measure is changed here.
Since $H>0$, the principal branch gives
$(HD_H)^{-\alpha}=H^{-\alpha}D_H^{-\alpha}$ when $D_H\ne0$.
Also $H/h=L$ and $\mueta(Q)H^{-\alpha}=A_Q$.
Thus, on $\mathcal R^2$,
\[
 F_Q^{(j)}
 =A_QD_H^{-\alpha}\psi(L|D_H|)
    \sum_{(m,n)\in\mathcal P_{j,Q}}
       \chi_{k,m}(x_Q+Ht)\chi_{k,n}(x_Q+Ht').
\]
The product is defined to be zero where the scale cutoff vanishes.
Since $\psi$ is supported in $[1/2,2]$, a nonzero value requires
\[
 \frac1{2L}\le |D_H|\le\frac2L.
\]
Thus the negative power is evaluated away from zero, with a lower
bound independent of $H$.
The extension to the larger box will be defined in Step~4.

For every integer $a\ge0$, the chain rule and the bounds for
$\chi_{k,m}$ give
\[
 \left|\frac{d^a}{dt^a}\chi_{k,m}(x_Q+Ht)\right|
 =H^a\left|\chi_{k,m}^{(a)}(x_Q+Ht)\right|
 \le C_a(H/h)^a=C_aL^a.
\]
The same estimate holds for the factor in $t'$.
Since $L=32$ is fixed and the sum has at most $M_0$ terms,
these angular derivative bounds are independent of $Q$.

\smallskip\noindent\emph{4. Extend the formula before taking its Fourier series.}
We construct a smooth function whose support lies in the interior of a
larger box. Its periodic extension will then be smooth.
All derivatives of functions on $\mathbb C$ below are real partial
derivatives, with $\mathbb C$ identified with $\mathbb R^2$.

Choose a smooth function $\beta:\mathbb R\to[0,1]$ such that
$\beta=0$ on $(-\infty,-1/2]$ and $\beta=1$ on $[0,\infty)$.
For $w\notin(-\infty,0]$, set
\[
 \Theta(w)=\beta\!\left(\frac{\operatorname{Re}w}{|w|}\right)
              w^{-\alpha}\psi(L|w|),
\]
where $w^{-\alpha}$ uses the principal branch.
Set $\Theta(w)=0$ for $w\in(-\infty,0]$.
This defines a smooth function on $\mathbb C$.
Indeed, $\psi(L|w|)=0$ when $|w|<(2L)^{-1}$.
Near each point of the negative real axis other than zero,
$\operatorname{Re}w/|w|<-1/2$, so the factor $\beta$ is zero.
At every remaining point, all factors are smooth.
Moreover,
\[
 \operatorname{supp}\Theta
 \subset\{w\in\mathbb C:(2L)^{-1}\le|w|\le2/L\}.
\]
Thus every real partial derivative of $\Theta$ is bounded.
On the closed right half-plane,
$\Theta(w)=w^{-\alpha}\psi(L|w|)$, with the product defined as zero
near the origin.

Choose a fixed $\omega\in C_c^\infty((-2,2)^2)$ equal to one on
$\mathcal R$. Define
\begin{align*}
 \widetilde F_Q^{(j)}(s,t,s',t')
 ={}&A_Q\omega(s,t)\omega(s',t')
       \Theta(D_H(s,t,s',t'))\\
 &\times\sum_{(m,n)\in\mathcal P_{j,Q}}
       \chi_{k,m}(x_Q+Ht)\chi_{k,n}(x_Q+Ht').
\end{align*}
On $\mathcal R^2$, both factors $\omega$ equal one and
$\operatorname{Re}D_H\ge0$. Hence
$\widetilde F_Q^{(j)}=F_Q^{(j)}$ there.
Also,
\[
 \operatorname{supp}\widetilde F_Q^{(j)}
 \subset\operatorname{supp}\omega\times\operatorname{supp}\omega,
\]
a fixed compact subset of $(-2,2)^4$.

The chain rule expresses each derivative of $\Theta(D_H)$ of order
at most six as a finite sum of products.
Each product contains a derivative of $\Theta$ evaluated at $D_H$
and derivatives of $D_H$ of order at most six.
All these factors have bounds independent of $H$.
Step~3 gives the same uniform bounds for the angular factors.
The sum has at most $M_0$ terms, and $A_Q$ is bounded by
\eqref{eq:hQ}. The product rule therefore gives
\begin{equation}\label{eq:uniform-CM}
 \sup_{\substack{j,Q\\Q\ne\D}}
 \|\widetilde F_Q^{(j)}\|_{C^6([-2,2]^4)}
 \le C_\eta.
\end{equation}
Here the $C^6$ norm is the maximum of the supremum norms of all
partial derivatives of total order at most six.
The cutoff functions and $L=32$ are fixed.
Thus $C_\eta$ does not depend on the height or the position of $Q$.

For the root, use Cartesian coordinates.
For $w\notin(-\infty,0]$, define
\[
 \Theta_0(w)
 =\beta\!\left(\frac{\operatorname{Re}w}{|w|}\right)
    w^{-\alpha}\psi_{\rm root}(|w|),
\]
using the principal branch, and set $\Theta_0=0$ on $(-\infty,0]$.
This function is smooth on $\mathbb C$.
Indeed, $\psi_{\rm root}$ vanishes near zero, and $\beta$ vanishes
near every other point of the branch cut.

Take a fixed $\omega_0\in C_c^\infty((-2,2)^2)$ equal to one on
the closed unit disc.
For $x=(x_1,x_2)$ and $y=(y_1,y_2)$, write
$z=x_1+ix_2$ and $\zeta=y_1+iy_2$, and set
\[
 \widetilde F_{\D}^{(1)}(x,y)
 =\omega_0(x)\omega_0(y)\Theta_0(1-\overline\zeta z).
\]
On $\D\times\D$,
\[
 \operatorname{Re}(1-\overline\zeta z)
 \ge1-|z||\zeta|>0.
\]
Both factors $\omega_0$ equal one there.
Since $\mueta(\D)=1$, this formula agrees with the root piece.
Its support is contained in
$\operatorname{supp}\omega_0\times\operatorname{supp}\omega_0$.
It is therefore a fixed smooth function of compact support in
$(-2,2)^4$, and its $C^6$ norm is finite.
The root terms in the other two trees are zero.
After enlarging $C_\eta$, \eqref{eq:uniform-CM} holds with the
supremum taken over all $j,Q$, including the roots.

\smallskip\noindent\emph{5. Bound all Fourier coefficients by one sequence.}
Extend each $\widetilde F_Q^{(j)}$ periodically, with period four
in each coordinate.
The extension is smooth because the function vanishes near the
boundary of the box.
For $r,q\in\mathbb Z^2$, define
\[
 c_{Q,(r,q)}^{(j)}
 =4^{-4}\int_{[-2,2]^2}\int_{[-2,2]^2}
      \widetilde F_Q^{(j)}(x,y)
      e^{-i\pi(r\cdot x+q\cdot y)/2}\,dx\,dy.
\]
Here $x=(x_1,x_2)$ and $y=(y_1,y_2)$ are the coordinate variables.
The factor $4^{-4}$ is the reciprocal of the volume of the box.
These coefficients use ordinary Lebesgue measure.
The resulting pointwise identity will later be integrated with
respect to $\mueta$.

Put $u=(x_1,x_2,y_1,y_2)$ and
$\xi=(r_1,r_2,q_1,q_2)$.
If $\xi\ne0$, choose $\ell$ such that
$|\xi_\ell|=\max_{1\le m\le4}|\xi_m|$.
Since $\xi$ has integer entries, $|\xi_\ell|\ge1$.
Six integrations by parts give
\[
 c_{Q,(r,q)}^{(j)}
 =4^{-4}\left(\frac{2}{i\pi\xi_\ell}\right)^6
   \int_{[-2,2]^4}
      \partial_{u_\ell}^{\,6}\widetilde F_Q^{(j)}(u)
      e^{-i\pi\xi\cdot u/2}\,du.
\]
All boundary terms vanish.
By \eqref{eq:uniform-CM},
\[
 |c_{Q,(r,q)}^{(j)}|
 \le C_\eta|\xi_\ell|^{-6}.
\]
Also $|r|+|q|\le2\sqrt2\,|\xi_\ell|$, where $|r|$ and $|q|$
are Euclidean norms.
Thus
\begin{equation}\label{eq:fourier-decay}
 |c_{Q,(r,q)}^{(j)}|
 \le C_\eta(1+|r|+|q|)^{-6}
\end{equation}
for every nonzero frequency.
At the zero frequency, the defining integral is bounded by
$\|\widetilde F_Q^{(j)}\|_\infty$.
After enlarging $C_\eta$, \eqref{eq:fourier-decay} holds for all
$j,Q,r,q$.

For each integer $n\ge0$, there are at most $C2^{4n}$ pairs
$(r,q)\in\mathbb Z^2\times\mathbb Z^2$ satisfying
\[
 2^n\le1+|r|+|q|<2^{n+1}.
\]
Indeed, each of the four integer coordinates has absolute value
less than $2^{n+1}$.
Hence \eqref{eq:fourier-decay} gives
\begin{equation}\label{eq:fourier-l1}
 \sum_{r,q\in\mathbb Z^2}\sup_{j,Q}|c_{Q,(r,q)}^{(j)}|
 \le C_\eta\sum_{n\ge0}2^{4n}2^{-6n}<\infty.
\end{equation}
Thus we may take
$C_{(r,q)}=C_\eta(1+|r|+|q|)^{-6}$ in
\eqref{eq:l1coeff}.

Since the exponential factors have modulus one, the Fourier series
converges absolutely and uniformly.
Its sum is continuous.
Termwise integration shows that this sum and the periodic extension
of $\widetilde F_Q^{(j)}$ have the same Fourier coefficients.
Their difference is zero in $L^2([-2,2]^4)$ by completeness of the
trigonometric system.
Since the difference is continuous, it is zero everywhere.
Restricting to the original coordinate set gives
\begin{equation}\label{eq:fourier-series}
 F_Q^{(j)}(x,y)
 =\sum_{r,q\in\mathbb Z^2}c_{Q,(r,q)}^{(j)}
      e^{i\pi r\cdot x/2}e^{i\pi q\cdot y/2}.
\end{equation}
For a nonroot tent, this set is $\mathcal R\times\mathcal R$.
For the root, $F_{\D}^{(1)}$ is the restriction of its Cartesian
extension to $\D\times\D$.

\smallskip\noindent\emph{6. Return to the disc and sum the pieces.}
For a nonroot tent, let $X_Q(z)=(s,t)$ be the coordinates from
Step~3, defined for $z\in Q$.
For the root, set
$X_{\D}(z)=(\operatorname{Re}z,\operatorname{Im}z)$.
Recall that $\mathcal N=\mathbb Z^2\times\mathbb Z^2$.
For $\nu=(r,q)\in\mathcal N$, define on $Q$
\[
 a_{Q,\nu}^{(j)}(z)=e^{i\pi r\cdot X_Q(z)/2},\qquad
 b_{Q,\nu}^{(j)}(z)=e^{i\pi q\cdot X_Q(z)/2},
\]
and set both functions equal to zero off $Q$.
They are measurable and have modulus at most one.
Set all coefficients equal to zero for a tent with no assigned piece.

Since $F_Q^{(j)}=\mueta(Q)K_Q^{(j)}$ in these coordinates,
\eqref{eq:fourier-series} gives
\[
 K_Q^{(j)}(z,\zeta)
 =\frac1{\mueta(Q)}
   \sum_{\nu\in\mathcal N}
       c_{Q,\nu}^{(j)}
       a_{Q,\nu}^{(j)}(z)b_{Q,\nu}^{(j)}(\zeta).
\]
This identity holds on $Q\times Q$.
It also holds off $Q\times Q$, where both sides are zero.
For each fixed $Q$, the series converges absolutely and uniformly,
since its terms are bounded by $C_\nu/\mueta(Q)$.

Fix $z\in\D$ away from the angular boundaries.
If a tent of height $H$ contains $z$, then
$H\ge1-|z|>0$.
The possible heights are dyadic, so only finitely many satisfy
this inequality.
At each height, each grid has at most one tent containing $z$.
Thus only finitely many tents can contribute for a fixed pair
$(z,\zeta)$.
The number of these tents may depend on $z$; no uniform bound on
that number is needed.

We may therefore sum the preceding identities over the three trees.
The assignment in Step~2 counts each localized piece exactly once.
Equation~\eqref{eq:smooth-decomp} then gives
\eqref{eq:separated-kernel}.
The angular boundaries form a countable union of sets of
$\mueta$-measure zero.
Hence the identity holds for almost every pair $(z,\zeta)$.
This proves the theorem.
\end{proof}

\begin{remark}\label{rem:factorization-meaning}
The coefficient estimate is uniform over all tents. This is stronger than separate absolute convergence on each tent.
It also justifies integration of the full expansion. Indeed, \eqref{eq:fourier-l1} and the dyadic heights give, for almost every fixed $z\in\D$,
\begin{align}\label{eq:separated-absolute}
 \sum_{j,Q,\nu}\frac{|c_{Q,\nu}^{(j)}|
           |a_{Q,\nu}^{(j)}(z)b_{Q,\nu}^{(j)}(\zeta)|}{\mueta(Q)}
 \le C_\eta\sum_{j=1}^3\sum_{Q\in\T_j:\,z\in Q}\frac1{\mueta(Q)}
 \le C_\eta(1-|z|)^{-\alpha}.
\end{align}
For the last inequality, use $\mueta(Q)\ge H^\alpha$ and sum the finite geometric series over $H\ge1-|z|$.
The bound is independent of $\zeta$. Thus one may integrate the series term by term against any Bochner integrable function of $\zeta$.
The absolute kernel estimate in Proposition~\ref{prop:convex-bergman} leads to a different sparse form.
\end{remark}

\subsection{A layer estimate for every exponent}\label{sec:layer-tree}
Write $M_Ef=\mathbf1_Ef$ for a measurable set $E$. For scalar $a$, write $M_af=af$. These multipliers are contractions when $|a|\le1$.

\begin{lemma}[The tent-layer estimate]\label{lem:layer-tree}
Let $1<p<\infty$ and let $\T$ be a full dyadic tent tree. For each $Q$, let $Q_1,Q_2$ be its children and put $G_Q=Q\setminus(Q_1\cup Q_2)$. Suppose that $0<\sigma<1$ and
$\mueta(G_Q)\ge\sigma\mueta(Q)$ for every $Q\in\T$.
Assume that $W$ satisfies the standing assumptions and
\[
 A=\sup_{Q\in\T}\|E_Q\|_{L^p(W)\to L^p(W)}<\infty.
\]
For measurable scalar functions $a_Q,b_Q$ that vanish off $Q$ and satisfy $|a_Q|,|b_Q|\le1$, put
\[
 Tf=\sum_{Q\in\T}a_Q\fint_Q b_Qf\,d\mueta.
\]
Then
\begin{equation}\label{eq:layer-tree-bound}
 \|T\|_{L^p(W)\to L^p(W)}
 \le \frac{A}{(1-\rho_p^{1/p})(1-\rho_{p'}^{1/p'})}
 \le pp'\sigma^{-p-p'}A^{1+p+p'},
 \qquad \rho_q=1-(\sigma/A)^q.
\end{equation}
Here $q$ takes the values $p$ and $p'$. The estimate holds uniformly for every finite subcollection of $\T$. It is independent of dimension.
\end{lemma}

\begin{proof}
Put $U=W^{1/p}$ and write
\[
 X=L^p(W),\qquad Y=L^{p'}(W^\#),\qquad W^\#=W^{-p'/p}.
\]
By \eqref{eq:dualweight}, $Y$ is the dual space of $X$ under the pairing
$\int_\D\langle f,g\rangle\,d\mueta$.
The norms use $U$ on $X$ and $U^{-1}$ on $Y$.
For a nonzero $e\in\Hh$, the function $\mathbf1_Qe$ belongs to $X$ and is fixed by $E_Q$.
Thus $A\ge1$ and $0<\rho_p,\rho_{p'}<1$.

\smallskip\noindent\emph{1. Estimate the mass left at each step.}
For $f\in X$ and $g\in Y$, set
$v_f=\fint_Qf\,d\mueta$ and $v_g=\fint_Qg\,d\mueta$.
These weak averages exist by Section~\ref{sec:prelim-scalar}.
Their definition gives
\[
 \int_\D\langle E_Qf,g\rangle\,d\mueta
 =\mueta(Q)\langle v_f,v_g\rangle
 =\int_\D\langle f,E_Qg\rangle\,d\mueta.
\]
Thus the dual operator of $E_Q$ is the same averaging map on $Y$.
In particular, $\|E_Q\|_{Y\to Y}=\|E_Q\|_{X\to X}\le A$.
This identity uses the unweighted pairing displayed above.

For a measurable $E\subset Q$ and $e\in\Hh$,
\[
 E_Q(\mathbf1_Ee)=\frac{\mueta(E)}{\mueta(Q)}\mathbf1_Qe.
\]
The standing assumptions make the following integrals finite. The bound
$\|E_Q\|_{X\to X}\le A$ gives
\begin{align*}
 \left(\frac{\mueta(E)}{\mueta(Q)}\right)^p
       \int_Q\|Ue\|^p\,d\mueta
 &=\|E_Q(\mathbf1_Ee)\|_X^p\\
 &\le A^p\|\mathbf1_Ee\|_X^p
 =A^p\int_E\|Ue\|^p\,d\mueta.
\end{align*}
Divide by $A^p$ and rearrange. We obtain
\begin{equation}\label{eq:layer-mass}
 \int_E\|Ue\|^p\,d\mueta
 \ge A^{-p}\left(\frac{\mueta(E)}{\mueta(Q)}\right)^p
       \int_Q\|Ue\|^p\,d\mueta.
\end{equation}
Applying the same argument on $Y$ gives this inequality with $U^{-1}$ and $p'$.
For $E=G_Q$, the measure assumption therefore yields
\[
 \int_{Q\setminus G_Q}\|Ue\|^p\,d\mueta
 \le\rho_p\int_Q\|Ue\|^p\,d\mueta,
 \qquad \rho_p=1-(\sigma/A)^p.
\]
The analogous bound for $U^{-1}$ has factor $\rho_{p'}$.

\smallskip\noindent\emph{2. Define the layers and check disjointness.}
All disjointness statements below are understood up to null sets.
Let $\mathcal C_j(Q)$ be the descendants exactly $j$ generations below $Q$.
Define
\[
 D_j(Q)=\bigcup_{R\in\mathcal C_j(Q)}R,\qquad
 D_0(Q)=Q,\qquad F_j(Q)=D_j(Q)\setminus D_{j+1}(Q).
\]
The children of the sets in $\mathcal C_j(Q)$ form $D_{j+1}(Q)$.
Since the sets in $\mathcal C_j(Q)$ are disjoint, the preceding bound gives
\begin{align*}
 \int_{D_{j+1}(Q)}\|Ue\|^p\,d\mueta
 &=\sum_{R\in\mathcal C_j(Q)}
        \int_{R\setminus G_R}\|Ue\|^p\,d\mueta\\
 &\le\rho_p\sum_{R\in\mathcal C_j(Q)}
        \int_R\|Ue\|^p\,d\mueta
 =\rho_p\int_{D_j(Q)}\|Ue\|^p\,d\mueta.
\end{align*}
Start with $D_0(Q)=Q$ and repeat this inequality. Apply the same argument to $U^{-1}$ with exponent $p'$. We obtain
\begin{align}
 \int_{D_j(Q)}\|Ue\|^p\,d\mueta
 &\le\rho_p^j\int_Q\|Ue\|^p\,d\mueta,\label{eq:layer-decay-p}\\
 \int_{D_k(Q)}\|U^{-1}e\|^{p'}\,d\mueta
 &\le\rho_{p'}^k\int_Q\|U^{-1}e\|^{p'}\,d\mueta.\label{eq:layer-decay-dual}
\end{align}
These estimates concern a fixed vector $e$. We will apply them to the constant value of an average.

We have
\[
 F_j(Q)=\bigcup_{R\in\mathcal C_j(Q)}G_R,
 \qquad \mueta(D_j(Q))\le(1-\sigma)^j\mueta(Q).
\]
Thus the sets $F_j(Q)$, for $j\ge0$, partition $Q$ up to a null set.
We also need a second disjointness fact: for a fixed $j$, the sets $F_j(Q)$ are pairwise disjoint as $Q$ varies over the whole tree.
Indeed, all top halves $G_R$ are pairwise disjoint. Each $R$ has at most one ancestor exactly $j$ generations above it. Hence the same top half cannot occur in both $F_j(Q)$ and $F_j(Q')$ when $Q\ne Q'$.

\smallskip\noindent\emph{3. Estimate one pair of layers.}
For $j,k\ge0$, put
\[
 B_{Q,j,k}=M_{F_j(Q)}M_{a_Q}E_QM_{b_Q}M_{F_k(Q)}.
\]
First restrict the argument of $E_Q$ to $F_k(Q)$.
For $g\in Y$, the function $E_Qg$ is constant on $Q$ and zero elsewhere.
Since $F_k(Q)\subset D_k(Q)$, \eqref{eq:layer-decay-dual} gives
\[
 \|M_{F_k(Q)}E_Qg\|_Y
 \le\rho_{p'}^{k/p'}\|E_Qg\|_Y
 \le A\rho_{p'}^{k/p'}\|g\|_Y.
\]
Use the dual identity from Step 1 and weighted H\"older's inequality. For $h\in X$,
\begin{align*}
 \|E_QM_{F_k(Q)}h\|_X
 &=\sup_{\|g\|_Y\le1}
      \left|\int_\D\langle h,M_{F_k(Q)}E_Qg\rangle\,d\mueta\right|\\
 &\le A\rho_{p'}^{k/p'}\|h\|_X.
\end{align*}
Now set $h=M_{F_k(Q)}f$ and $v=E_QM_{b_Q}h$.
Because $h$ vanishes off $F_k(Q)$ and $b_Q$ is scalar,
\[
 v=E_QM_{F_k(Q)}M_{b_Q}h,\qquad
 \|v\|_X\le A\rho_{p'}^{k/p'}\|h\|_X.
\]
Write $v=\mathbf1_Qe$ for its constant value $e\in\Hh$.
Using $|a_Q|\le1$ and \eqref{eq:layer-decay-p}, we get
\[
 \|M_{F_j(Q)}M_{a_Q}v\|_X^p
 \le\int_{F_j(Q)}\|Ue\|^p\,d\mueta
 \le\rho_p^j\|v\|_X^p.
\]
Consequently
\begin{equation}\label{eq:layer-block}
 \|B_{Q,j,k}f\|_X
 \le A\rho_p^{j/p}\rho_{p'}^{k/p'}
       \|M_{F_k(Q)}f\|_X.
\end{equation}
There is only one factor $A$. The second restriction uses the constant value $e$ and adds only the decay factor.

\smallskip\noindent\emph{4. Sum over a finite family.}
Let $\mathcal F\subset\T$ be finite and put
$T_{j,k}^{\mathcal F}=\sum_{Q\in\mathcal F}B_{Q,j,k}$.
For fixed $j$, the values of these summands lie on disjoint sets $F_j(Q)$.
Equation~\eqref{eq:layer-block} therefore gives
\begin{align*}
 \|T_{j,k}^{\mathcal F}f\|_X^p
 &=\sum_{Q\in\mathcal F}\|B_{Q,j,k}f\|_X^p\\
 &\le A^p\rho_p^j\rho_{p'}^{kp/p'}
        \sum_{Q\in\mathcal F}\|M_{F_k(Q)}f\|_X^p\\
 &\le A^p\rho_p^j\rho_{p'}^{kp/p'}\|f\|_X^p.
\end{align*}
The last step uses disjointness for the fixed layer number $k$.
Hence the double series converges in operator norm, and
\[
 \sum_{j,k\ge0}\|T_{j,k}^{\mathcal F}\|_{X\to X}
 \le\frac{A}{(1-\rho_p^{1/p})(1-\rho_{p'}^{1/p'})}.
\]

We next identify its sum. Put
\[
 P_{Q,J}=\sum_{j=0}^JM_{F_j(Q)}
        =M_{Q\setminus D_{J+1}(Q)},\qquad
 S_Q=M_{a_Q}E_QM_{b_Q}.
\]
For every $f\in X$, dominated convergence gives
\[
 \|(M_Q-P_{Q,J})f\|_X^p
 =\int_{D_{J+1}(Q)}\|Uf\|^p\,d\mueta\longrightarrow0.
\]
Also $S_Q=M_QS_QM_Q$ and $\|S_Q\|\le A$.
Thus
\begin{align*}
 \|P_{Q,J}S_QP_{Q,K}f-S_Qf\|_X
 &\le A\|(P_{Q,K}-M_Q)f\|_X\\
 &\quad+\|(P_{Q,J}-M_Q)S_Qf\|_X
 \longrightarrow0.
\end{align*}
But the expression $P_{Q,J}S_QP_{Q,K}$ is precisely
$\sum_{j=0}^J\sum_{k=0}^KB_{Q,j,k}$.
Sum over the finite family $\mathcal F$. The double series therefore equals $\sum_{Q\in\mathcal F}S_Q$.
This proves the first bound in \eqref{eq:layer-tree-bound} for every finite family.

\smallskip\noindent\emph{5. Pass to the whole tree and compute the power.}
Let $T_N$ be the sum over generations $0,\ldots,N$.
For a fixed $z\in\D$, a tent containing $z$ has height at least $1-|z|$.
Only finitely many generations can contain that point. Thus $T_Nf(z)$ is eventually equal to the displayed sum defining $Tf(z)$, for almost every $z$.
Fatou's lemma applied to $\|U(z)T_Nf(z)\|^p$ gives the same norm bound for $T$.
This step requires no operator-norm convergence of the generation sums.

For $0\le x\le1$ and $q>1$, concavity gives
\[
 1-(1-x)^{1/q}\ge x/q.
\]
Taking $x=(\sigma/A)^q$ yields
\[
 \frac1{1-\rho_q^{1/q}}\le q\sigma^{-q}A^q.
\]
Use this estimate with $q=p$ and $q=p'$. Multiplying the two factors and the initial factor $A$ gives
$pp'\sigma^{-p-p'}A^{1+p+p'}$.
Every constant used above depends only on $p$ and $\sigma$.
\end{proof}

\begin{proof}[\textbf{Proof of Theorem~\ref{thm:operator-main}}]
Write $X=L^p(W)$ and $U=W^{1/p}$.
If $P_\eta$ is bounded, Proposition~\ref{prop:matrix-necessity} gives
\[
 B^{1/p}=\sup_S\|E_S\|_{X\to X}
 \le C_\eta\|P_\eta\|_{X\to X}.
\]
This proves necessity and the lower bound.
We prove sufficiency assuming $B<\infty$. Put $A=B^{1/p}\ge1$.

\smallskip\noindent\emph{1. Bound each tree operator.}
Use Theorem~\ref{thm:separated-factorization}.
For each $\nu$ with $C_\nu>0$, define
\[
 \widehat a_{Q,\nu}^{(j)}
 =\frac{c_{Q,\nu}^{(j)}}{C_\nu}a_{Q,\nu}^{(j)},
 \qquad
 R_{j,\nu}
 =\sum_{Q\in\T_j}M_{\widehat a_{Q,\nu}^{(j)}}
             E_QM_{b_{Q,\nu}^{(j)}}.
\]
Both scalar factors have modulus at most one and vanish off $Q$.
Every averaging operator in the tree has norm at most $A$.
By Lemma~\ref{lem:full-tree-sparse}, all three trees have the same positive lower bound $\sigma_\eta$ for the relative measure of $G_Q$.
Lemma~\ref{lem:layer-tree} therefore gives
\[
 \|R_{j,\nu}\|_{X\to X}\le C_{p,\eta}A^{1+p+p'}
\]
uniformly in $j,\nu$ and independently of dimension.
If $C_\nu=0$, all corresponding coefficients are zero, so that index may be omitted.

\smallskip\noindent\emph{2. Sum in operator norm.}
Since $\sum_\nu C_\nu<\infty$, the series
\[
 S=\sum_{j=1}^3\sum_{\nu:C_\nu>0}C_\nu R_{j,\nu}
\]
converges in the space of bounded operators on $X$.
Its norm satisfies
\[
 \|S\|_{X\to X}
 \le3C_{p,\eta}A^{1+p+p'}\sum_\nu C_\nu
 \le C_{p,\eta}B^{\gamma_p},
\]
because
\[
 \frac{1+p+p'}p=1+\frac1p+\frac1{p-1}=\gamma_p.
\]
The coefficient sum depends only on $\eta$ and the fixed construction in Theorem~\ref{thm:separated-factorization}.

\smallskip\noindent\emph{3. Identify $S$ with the kernel operator.}
Take $f\in\mathscr D_W$. It is compactly supported and Bochner integrable by Section~\ref{sec:prelim-scalar}.
Choose increasing finite sets $\mathcal N_m$ whose union is $\mathcal N$, and let $S_m$ be the corresponding partial sums defining $S$.
For almost every fixed $z$, \eqref{eq:separated-absolute} gives
\begin{align*}
 &\sum_{j,Q,\nu}\frac{|c_{Q,\nu}^{(j)}|
            |a_{Q,\nu}^{(j)}(z)|}{\mueta(Q)}
       \int_Q|b_{Q,\nu}^{(j)}(\zeta)|\,\|f(\zeta)\|\,d\mueta(\zeta)\\
 &\qquad\le C_\eta(1-|z|)^{-\alpha}
       \int_\D\|f(\zeta)\|\,d\mueta(\zeta)<\infty.
\end{align*}
This permits termwise integration of the kernel expansion.
Equation~\eqref{eq:separated-kernel} therefore gives
$S_mf(z)\to P_\eta f(z)$ almost everywhere.

On the other hand, $S_mf\to Sf$ in $X$ by operator-norm convergence.
Equivalently, $US_mf\to USf$ in ordinary $L^p(\mueta;\Hh)$.
There is a subsequence converging almost everywhere in $\Hh$.
Since $U(z)^{-1}$ is bounded at almost every fixed $z$, this subsequence also satisfies $S_mf(z)\to Sf(z)$ there.
Comparison with the preceding pointwise limit gives $Sf=P_\eta f$ almost everywhere.

Thus $S$ extends the kernel operator from the dense class $\mathscr D_W$ to $X$.
The extension is unique, and its norm has the required upper bound.
The two bounds prove the theorem.
\end{proof}

\subsection{The Hilbert exponent}\label{sec:hilbert-bound}
The following argument adapts the almost-orthogonality method of \cite[Theorem~2.1]{LimaniPott21} to the full tent tree. We give the estimates because the children of a tent do not cover their parent.

\begin{lemma}\label{lem:hilbert-tree}
Let $\T$ be a full dyadic tent tree and let $W$ satisfy the standing assumptions for $p=2$. Suppose
\[
 B_\T:=\sup_{Q\in\T}\|E_Q\|_{L^2(W)\to L^2(W)}^2<\infty.
\]
For measurable scalar functions $a_Q,b_Q$ supported in $Q$ with $|a_Q|,|b_Q|\le1$, the operator
\[
 Tf=\sum_{Q\in\T}a_Q\fint_Q b_Qf\,d\mueta
\]
is bounded on $L^2(W)$, with $\|T\|\le C_\eta B_\T^{3/2}$. The same bound holds for every finite subcollection of $\T$. The constant is independent of $\dim\Hh$.
\end{lemma}

\begin{proof}
Write $B=B_\T$, $U=W^{1/2}$, and
\[
 X=L^2(W),\qquad Y=L^2(W^{-1}),\qquad
 \mathcal L=L^2(\mueta;\Hh).
\]
All unmarked operator norms below are on $\mathcal L$.
The function $\mathbf1_Qe$, with $e\ne0$, belongs to $X$ and is fixed by $E_Q$.
Thus $B\ge1$.

\smallskip\noindent\emph{1. Estimate the mass on descendants.}
The dual identity in the proof of Lemma~\ref{lem:layer-tree} gives
\[
 \|E_Q\|_{Y\to Y}=\|E_Q\|_{X\to X}\le\sqrt B.
\]
Hence $W^{-1}$ has the same averaging characteristic on the tree.
Equation~\eqref{eq:layer-mass}, with $p=2$ and $A=\sqrt B$, gives
\begin{equation}\label{eq:hilbert-mass}
 \int_E\|U^{\pm1}e\|^2\,d\mueta
 \ge\frac1B\left(\frac{\mueta(E)}{\mueta(Q)}\right)^2
       \int_Q\|U^{\pm1}e\|^2\,d\mueta
 \qquad(E\subset Q,\ e\in\Hh).
\end{equation}
The two signs denote two separate inequalities.
Let $\kappa_\eta$ be as in Lemma~\ref{lem:full-tree-sparse}, and put
\[
 \sigma=1-\kappa_\eta\in(0,1),\qquad
 \rho=1-\sigma^2/B\in(0,1).
\]
For the top half $G_Q$, we have $\mueta(G_Q)\ge\sigma\mueta(Q)$.
Taking $E=G_Q$ in \eqref{eq:hilbert-mass} and subtracting from the integral over $Q$ gives
\[
 \int_{Q\setminus G_Q}\|U^{\pm1}e\|^2\,d\mueta
 \le\rho\int_Q\|U^{\pm1}e\|^2\,d\mueta.
\]
Let $D_k(Q)$ be the union of the descendants exactly $k$ generations below $Q$, with $D_0(Q)=Q$.
Sum the last inequality over each generation and iterate. This gives
\begin{equation}\label{eq:hilbert-decay}
 \int_{D_k(Q)}\|U^{\pm1}e\|^2\,d\mueta
 \le\rho^k\int_Q\|U^{\pm1}e\|^2\,d\mueta.
\end{equation}
Here $e$ is a fixed vector. We shall use this estimate on the constant value of an average.

\smallskip\noindent\emph{2. Define the conjugated operators and their adjoints.}
The maps
\[
 J_+:X\to\mathcal L,\quad J_+f=Uf,
 \qquad J_-:Y\to\mathcal L,\quad J_-g=U^{-1}g
\]
are onto isometries. For each tent, define
\[
 S_Q=M_{a_Q}E_QM_{b_Q}\quad\hbox{on }X,
 \qquad L_Q=J_+S_QJ_+^{-1}\quad\hbox{on }\mathcal L.
\]
Scalar multiplication commutes with $U$ and is contractive when its modulus is at most one.
It follows that $\|L_Q\|\le\sqrt B$ and $L_Q=M_QL_QM_Q$.

We verify the adjoint on $\mathcal L$ by pairing.
For $F,H\in\mathcal L$, put $f=U^{-1}F\in X$ and $g=UH\in Y$.
The weak-average identity gives
\begin{align*}
 \langle L_QF,H\rangle_{\mathcal L}
 &=\int_\D\langle S_Qf,g\rangle\,d\mueta\\
 &=\int_\D\langle f,
      M_{\overline{b_Q}}E_QM_{\overline{a_Q}}g\rangle\,d\mueta\\
 &=\left\langle F,
      J_-M_{\overline{b_Q}}E_QM_{\overline{a_Q}}J_-^{-1}H
      \right\rangle_{\mathcal L}.
\end{align*}
Thus
\begin{equation}\label{eq:hilbert-adjoint}
 L_Q^*=J_-M_{\overline{b_Q}}E_QM_{\overline{a_Q}}J_-^{-1}.
\end{equation}
The middle three maps act on $Y$ and have product norm at most $\sqrt B$.
This argument does not assume that multiplication by $U$ or $U^{-1}$ is bounded on $\mathcal L$ itself.
More explicitly, on $Q$ these operators have the form
\begin{align*}
 L_QF(z)&=a_Q(z)U(z)\fint_Q b_QU^{-1}F\,d\mueta,\\
 L_Q^*H(z)&=\overline{b_Q(z)}U(z)^{-1}
             \fint_Q\overline{a_Q}UH\,d\mueta.
\end{align*}
Both averages are weak averages. Each is a fixed vector on $Q$.
This is why the directional decay estimate applies in the next step.

\smallskip\noindent\emph{3. Compare two generations.}
Let $\T_n$ be generation $n$ of the tree, starting with the root at $n=0$, and put
\[
 T_n=\sum_{Q\in\T_n}S_Q,\qquad
 L_n=J_+T_nJ_+^{-1}=\sum_{Q\in\T_n}L_Q.
\]
The sets in $\T_n$ are pairwise disjoint up to null sets.
Both $L_Q$ and $L_Q^*$ are supported in $Q$ and depend only on the restriction to $Q$.
Therefore, for $H\in\mathcal L$,
\[
 \|L_n^*H\|_2^2
 =\sum_{Q\in\T_n}\|L_Q^*M_QH\|_2^2
 \le B\sum_{Q\in\T_n}\|M_QH\|_2^2
 \le B\|H\|_2^2.
\]
The same calculation applies to $L_n$. In particular, $\|L_n\|\le\sqrt B$.

Fix $n>m$ and $F\in\mathcal L$.
For $R\in\T_m$, let
\[
 e_R=\fint_R b_RU^{-1}F\,d\mueta.
\]
This weak average exists by Section~\ref{sec:prelim-scalar}.
On $R$, we have $L_mF=a_RUe_R$.
Each tent in $\T_n$ lies in a unique $R\in\T_m$.
Consequently, $|a_R|\le1$ and \eqref{eq:hilbert-decay} give
\begin{align*}
 \|L_n^*L_mF\|_2^2
 &\le B\sum_{Q\in\T_n}\|M_QL_mF\|_2^2\\
 &=B\sum_{R\in\T_m}\int_{D_{n-m}(R)}
                   |a_R|^2\|Ue_R\|^2\,d\mueta\\
 &\le B\rho^{n-m}\sum_{R\in\T_m}
                      \|\mathbf1_RUe_R\|_2^2.
\end{align*}
We now estimate $\sum_{R\in\T_m}\|\mathbf1_RUe_R\|_2^2$.
The local averaging bound gives
\[
 \|\mathbf1_RUe_R\|_2
 =\|E_RM_{b_R}U^{-1}M_RF\|_X
 \le\sqrt B\|M_RF\|_2.
\]
The sets $R\in\T_m$ are disjoint, so
\[
 \|L_n^*L_mF\|_2^2\le B^2\rho^{n-m}\|F\|_2^2.
\]
Taking square roots gives $\|L_n^*L_m\|\le B\rho^{(n-m)/2}$.

By \eqref{eq:hilbert-adjoint}, the sequence $L_n^*$ has the same form with weight $W^{-1}$ and scalar factors $\overline{b_Q},\overline{a_Q}$.
The averaging bound and the decay factor $\rho$ are unchanged by Step 1.
Apply the preceding calculation to this sequence. It gives
$\|L_nL_m^*\|\le B\rho^{(n-m)/2}$ for $n>m$.
Taking adjoints proves both estimates when $n<m$.
For $n=m$, use $\|L_n\|^2\le B$.
We have proved
\begin{equation}\label{eq:hilbert-cross}
 \max\{\|L_n^*L_m\|,\|L_nL_m^*\|\}
 \le B\rho^{|n-m|/2}\qquad(n,m\ge0).
\end{equation}

\smallskip\noindent\emph{4. Sum the generation estimates.}
The Cotlar--Stein lemma applies to a finite sequence on a Hilbert space.
If both cross products have norm at most $h(n-m)^2$ and
$\sum_{k\in\mathbb Z}h(k)<\infty$, it gives
$\|\sum_nL_n\|\le2\sum_{k\in\mathbb Z}h(k)$;
see \cite[Lemma~3.1]{LimaniPott21}.
In \eqref{eq:hilbert-cross}, take
\[
 h(k)=\sqrt B\,\rho^{|k|/4}.
\]
The exponent is divided by two because the lemma sums square roots of the cross-product bounds.
For every $N\ge0$, it follows that
\[
 \left\|\sum_{n=0}^N L_n\right\|
 \le2\sqrt B\sum_{k\in\mathbb Z}\rho^{|k|/4}
 =2\sqrt B\,\frac{1+\rho^{1/4}}{1-\rho^{1/4}}
 \le\frac{4\sqrt B}{1-\rho^{1/4}}.
\]
Since $1-\rho^{1/4}\ge(1-\rho)/4=\sigma^2/(4B)$, this yields
\[
 \left\|\sum_{n=0}^N T_n\right\|_{X\to X}
 =\left\|\sum_{n=0}^N L_n\right\|
 \le16\sigma^{-2}B^{3/2}.
\]
Thus the factor $B^{1/2}$ comes from one generation. The geometric sum contributes one further factor $B$.

\smallskip\noindent\emph{5. Pass to the whole tree.}
For a fixed $z\in\D$, a tent containing $z$ has height at least $1-|z|$.
Only finitely many generations can contain $z$.
Hence $\sum_{n=0}^NT_nf(z)$ is eventually equal to $Tf(z)$ for almost every $z$.
Fatou's lemma gives
\[
 \|Tf\|_X^2
 \le\liminf_{N\to\infty}
       \left\|\sum_{n=0}^NT_nf\right\|_X^2
 \le (16\sigma^{-2}B^{3/2})^2\|f\|_X^2.
\]
This proves the assertion, since $\sigma$ depends only on $\eta$.
Setting $a_Q=0$ off a prescribed finite subcollection proves the same bound for that subcollection.
No estimate above depends on dimension.
\end{proof}

\begin{proof}[\textbf{Proof of Theorem~\ref{thm:hilbert-bound}}]
Put $X=L^2(W)$ and $U=W^{1/2}$.
For each of the three trees in Theorem~\ref{thm:separated-factorization},
\[
 B_{\T_j}=\sup_{Q\in\T_j}\|E_Q\|_{X\to X}^2\le B.
\]
For $C_\nu>0$, set
\[
 \widehat a_{Q,\nu}^{(j)}
 =\frac{c_{Q,\nu}^{(j)}}{C_\nu}a_{Q,\nu}^{(j)},
 \qquad
 R_{j,\nu}=\sum_{Q\in\T_j}
       M_{\widehat a_{Q,\nu}^{(j)}}E_QM_{b_{Q,\nu}^{(j)}}.
\]
The two scalar factors vanish off $Q$ and have modulus at most one.
Lemma~\ref{lem:hilbert-tree} therefore gives
$\|R_{j,\nu}\|_{X\to X}\le C_\eta B^{3/2}$, uniformly in $j,\nu$.
Indices with $C_\nu=0$ contribute nothing.
Since $\sum_\nu C_\nu<\infty$, the series
\[
 S=\sum_{j=1}^3\sum_{\nu:C_\nu>0}C_\nu R_{j,\nu}
\]
converges in operator norm on $X$, and
\[
 \|S\|_{X\to X}
 \le3C_\eta B^{3/2}\sum_\nu C_\nu
 \le C_\eta B^{3/2}.
\]
The coefficient sum depends only on $\eta$ and the fixed kernel construction.

We identify $S$ with $P_\eta$ on the dense class $\mathscr D_W$.
Take $f\in\mathscr D_W$ and let $S_m$ be the partial sum over increasing finite coefficient sets exhausting $\mathcal N$.
The function $f$ is Bochner integrable.
Equation~\eqref{eq:separated-absolute} bounds the integral of the absolute series against $\|f\|$ by
\[
 C_\eta(1-|z|)^{-\alpha}\int_\D\|f(\zeta)\|\,d\mueta(\zeta)<\infty
\]
for almost every fixed $z$.
Thus termwise integration in \eqref{eq:separated-kernel} gives
$S_mf(z)\to P_\eta f(z)$ almost everywhere.
On the other hand, $S_mf\to Sf$ in $X$.
Hence $US_mf\to USf$ in ordinary $L^2$, and a subsequence converges almost everywhere.
Multiplying by the bounded operator $U(z)^{-1}$ at each such point gives $Sf=P_\eta f$ almost everywhere.
Density now proves \eqref{eq:hilbert-upper}.
All constants are independent of dimension.
\end{proof}

\section{Consequences and examples}\label{sec:examples}
\subsection{A noncommuting example}
\begin{proposition}\label{prop:noncommuting-example}
Let $\mathcal K$ be a nonzero separable Hilbert space and set $\Hh=\mathcal K\oplus\mathcal K$. Choose
\[
 -(\eta+1)<a<0<b<(p-1)(\eta+1),\qquad t(z)=1-|z|^2.
\]
On $\Hh$, define
\[
 C=\begin{pmatrix}I&I\\0&I\end{pmatrix},\qquad
 D(z)=\begin{pmatrix}t(z)^{a/p}I&0\\0&t(z)^{b/p}I\end{pmatrix},
 \qquad U(z)=C^*D(z)C,\qquad W(z)=U(z)^p.
\]
Then $W$ satisfies the standing assumptions and $P_\eta$ is bounded on $L^p(W)$. The values of $W$ at distinct radii do not commute. There is no scalar weight $w$ with $c wI\le W\le C_0 wI$ for fixed positive $c,C_0$. In particular, this gives an infinite-dimensional example when $\dim\mathcal K=\infty$.
\end{proposition}

\begin{proof}
Since $t(z)>0$ in $\D$, the operator $D(z)$ is positive and boundedly invertible.
The same holds for $U(z)=C^*D(z)C$, and
\[
 U(z)^{-1}=C^{-1}D(z)^{-1}(C^*)^{-1}.
\]
All these operators depend continuously on $z$.
For every $e\in\Hh$, the two formulas give
\begin{align*}
 \|U(z)e\|^p
 &\le C_p\big(t(z)^a+t(z)^b\big)\|e\|^p,\\
 \|U(z)^{-1}e\|^{p'}
 &\le C_p\big(t(z)^{-a/(p-1)}+t(z)^{-b/(p-1)}\big)\|e\|^{p'}.
\end{align*}
Here the constants also depend on the fixed matrix $C$, but not on $\dim\mathcal K$.
Polar integration gives
\[
 \int_\D t(z)^\beta\,d\mueta(z)
 =\frac{\eta+1}{\eta+1+\beta}
 \qquad\big(\beta>-(\eta+1)\big).
\]
All four powers above satisfy this condition.
Thus $W^{1/p}=U$ satisfies both standing integrability assumptions.

For $F=Cf=(F_1,F_2)$, invertibility of $C^*$ gives
\[
 \|Uf\|^p=\|C^*DF\|^p
 \asymp_{p,C}\big(t^{2a/p}\|F_1\|^2+t^{2b/p}\|F_2\|^2\big)^{p/2}
 \asymp_p t^a\|F_1\|^p+t^b\|F_2\|^p.
\]
The last comparison is the elementary equivalence of the two norms on $\mathbb R^2$.
After integration, this yields
\[
 \|f\|_{L^p(W)}^p
 \asymp_{p,C}\int_\D
 \big(t^a\|F_1\|_{\mathcal K}^p+t^b\|F_2\|_{\mathcal K}^p\big)\,d\mueta.
\]
By Propositions~\ref{prop:power} and \ref{thm:scalar-vector}, the projection is bounded on both $L^p(t^a\,d\mueta;\mathcal K)$ and $L^p(t^b\,d\mueta;\mathcal K)$.
Apply these two operators to the coordinates of $F$, then multiply by $C^{-1}$.
The displayed norm comparison gives a bounded operator on $L^p(W)$.
On $\mathscr D_W$, the kernel commutes with the constant operator $C$.
Thus this operator agrees there with $P_\eta$.
Density proves the required boundedness, with constants independent of $\dim\mathcal K$.
Proposition~\ref{prop:matrix-necessity} also gives the averaging condition.

To check noncommutation, fix a unit vector $v\in\mathcal K$.
The space spanned by $(v,0)$ and $(0,v)$ is invariant under all the operators above.
Write $x=t^{a/p}$ and $y=t^{b/p}$. On this two-dimensional space,
\[
 U(t)=x\begin{pmatrix}1&1\\1&1\end{pmatrix}
      +y\begin{pmatrix}0&0\\0&1\end{pmatrix}.
\]
If the two fixed matrices are $A_0,B_0$, then
\[
 [A_0,B_0]=\begin{pmatrix}0&1\\-1&0\end{pmatrix},
 \qquad [U(t_1),U(t_2)]=(x_1y_2-x_2y_1)[A_0,B_0].
\]
Here $[A,B]=AB-BA$. The ratio $x/y=t^{(a-b)/p}$ is strictly monotone. Hence $U(t_1)$ and $U(t_2)$ do not commute for $t_1\ne t_2$. If their $p$th powers commuted, continuous functional calculus would imply that their positive $p$th roots commute, a contradiction.

Finally, let $\lambda_+\ge\lambda_->0$ be the eigenvalues of this block of $U(t)$.
Its determinant is $xy$, and its largest eigenvalue is at least the first diagonal entry $x$.
Hence
\[
 \frac{\lambda_+}{\lambda_-}
 =\frac{\lambda_+^2}{xy}\ge\frac{x}{y}.
\]
The corresponding eigenvalues of $W(t)$ are $\lambda_+^p,\lambda_-^p$.
Their ratio is at least $(x/y)^p=t^{a-b}$, which tends to infinity as $t\downarrow0$.
If $cw(z)I\le W(z)\le C_0w(z)I$, this ratio would be at most $C_0/c$ almost everywhere.
The lower bound exceeds any fixed constant on a sufficiently thin boundary annulus.
This is a contradiction.
\end{proof}

This example gives noncommuting weights with unbounded anisotropy. Its boundedness follows from a constant change of coordinates and two scalar theorems. It illustrates the theorem but requires no infinite-dimensional argument of its own.

\subsection{Changes on compact subsets}\label{sec:compact-changes}
For a weight $W$, define the finite constants
\[
 C_+(W)=\sup_{\|e\|=1}\left(\int_\D\|W^{1/p}e\|^p\,d\mueta\right)^{1/p},
 \qquad
 C_-(W)=\sup_{\|e\|=1}\left(\int_\D\|W^{-1/p}e\|^{p'}\,d\mueta\right)^{1/p'}.
\]
Their finiteness follows from \eqref{eq:directional-uniform-local} at the root. No integrability of the operator norms is assumed.

\begin{proposition}\label{prop:compact-stability}
Suppose that $W$ and $V$ satisfy the standing assumptions and agree on $\{r<|z|<1\}$ for some $0<r<1$. Then $P_\eta$ is bounded on $L^p(W)$ if and only if it is bounded on $L^p(V)$. More precisely,
\begin{equation}\label{eq:compact-transfer}
 \|P_\eta\|_{L^p(W)\to L^p(W)}
 \le\|P_\eta\|_{L^p(V)\to L^p(V)}
       +2(1-r)^{-\eta-2}C_+(W)C_-(W),
\end{equation}
whenever the norm on the right is finite. The averaging characteristic is finite for $W$ if and only if it is finite for $V$.
\end{proposition}

\begin{proof}
Put $X=L^p(W)$, $U=W^{1/p}$, $K=\{|z|\le r\}$, $F=\D\setminus K$, and $\alpha=\eta+2$.
We first construct two bounded operators on $X$ from the scalar expansion
\[
 (1-\overline\zeta z)^{-\alpha}
 =\sum_{n\ge0}a_nz^n\overline\zeta^{\,n},
 \qquad a_n=\frac{\alpha(\alpha+1)\cdots(\alpha+n-1)}{n!},\quad a_0=1.
\]
For $e\in\Hh$ and $f\in X$, define
\[
 (J_ne)(z)=z^ne,\qquad
 \Lambda_nf=\int_\D\overline\zeta^{\,n}f(\zeta)\,d\mueta(\zeta).
\]
The second integral is weak. For $h\in\Hh$, weighted H\"older's inequality gives
\[
 |\langle\Lambda_nf,h\rangle|
 \le\|f\|_X
      \left(\int_\D |\zeta|^{np'}\|U(\zeta)^{-1}h\|^{p'}\,d\mueta\right)^{1/p'}
 \le C_-(W)\|f\|_X\|h\|.
\]
Thus $\Lambda_n:X\to\Hh$ is bounded. Also $\|J_n\|_{\Hh\to X}\le C_+(W)$.
Restricting to $K$ gives
\[
 \|\Lambda_nM_K\|_{X\to\Hh}\le r^nC_-(W),
 \qquad \|M_KJ_n\|_{\Hh\to X}\le r^nC_+(W).
\]
The two series
\[
 R_K=\sum_{n\ge0}a_nJ_n\Lambda_nM_K,
 \qquad L_K=\sum_{n\ge0}a_nM_KJ_n\Lambda_n
\]
therefore converge in operator norm on $X$. Since $a_n>0$, their norms satisfy
\[
 \|R_K\|,\ \|L_K\|
 \le C_+(W)C_-(W)\sum_{n\ge0}a_nr^n
 =(1-r)^{-\alpha}C_+(W)C_-(W).
\]

For $f\in\mathscr D_W$, the function $f$ is Bochner integrable.
When $z\in K$ or $\zeta\in K$, the absolute scalar series is at most $(1-r)^{-\alpha}$.
This permits termwise integration in the two restricted series.
Their operator-norm sums agree with their pointwise sums almost everywhere: multiply the partial sums by $U$ and take an almost everywhere convergent subsequence in ordinary $L^p$.
Consequently,
\[
 R_Kf(z)=\int_K K_\eta(z,\zeta)f(\zeta)\,d\mueta(\zeta),
 \qquad
 L_Kf(z)=\mathbf1_K(z)\int_\D K_\eta(z,\zeta)f(\zeta)\,d\mueta(\zeta).
\]

Now assume that the bounded extension $P_{\eta,V}$ on $L^p(V)$ exists.
Because $W=V$ on $F$, restriction to $F$ is a contraction from $X$ to $L^p(V)$ and from $L^p(V)$ to $X$.
Hence
\[
 A_F=M_FP_{\eta,V}M_F:X\to X,
 \qquad\|A_F\|_{X\to X}\le\|P_{\eta,V}\|.
\]
If $f=W^{-1/p}f_0\in\mathscr D_W$, then
$M_Ff=V^{-1/p}(M_Ff_0)\in\mathscr D_V$.
Thus $A_F$ agrees with the exterior part of the kernel on $\mathscr D_W$.
On this dense class, the bounded operator
\[
 S=A_F+R_K+L_KM_F
\]
equals the full kernel integral. Indeed,
\[
 1=\mathbf1_F(z)\mathbf1_F(\zeta)
       +\mathbf1_K(\zeta)+\mathbf1_K(z)\mathbf1_F(\zeta).
\]
It follows that $S$ is the bounded extension of $P_\eta$ to $X$.
The three norm estimates prove \eqref{eq:compact-transfer}.
Interchanging $W$ and $V$ proves the converse.

Finally, consider the averaging characteristic directly.
If $S=S(I)$ has $|I|<1-r$, then $S\subset F$.
The two weighted norms agree on functions supported in $S$, so the norm of $E_S$ is the same for $W$ and $V$.
If $|I|\ge1-r$, the measure formula in Lemma~\ref{lem:full-tree-sparse} gives $\mueta(S)\ge(1-r)^\alpha$.
The weak average and the constant embedding have product norm at most
\[
 \|E_S\|_{X\to X}
 \le\frac{C_+(W)C_-(W)}{\mueta(S)}
 \le(1-r)^{-\alpha}C_+(W)C_-(W).
\]
The same estimate holds for $V$.
Thus the small squares have identical averaging norms, and the large squares are uniformly bounded for both weights.
This proves the last assertion.
\end{proof}

\begin{proposition}\label{prop:no-global-RH}
For every $1<p<\infty$, $\eta>-1$, and nonzero separable Hilbert space $\mathcal K$, there is a weight on $\mathcal H=\mathcal K\oplus\mathcal K$ with the following properties. The weight is noncommuting and is not uniformly scalar-comparable. The projection $P_\eta$ is bounded. There are no $\varepsilon>0$ and $R<\infty$ such that
\[
 \left(\fint_S\|W^{1/p}e\|^{p+\varepsilon}\,d\mueta\right)^{1/(p+\varepsilon)}
 \le R\left(\fint_S\|W^{1/p}e\|^p\,d\mueta\right)^{1/p}
\]
for all Carleson squares $S$ and all $e\in\Hh$. The space $\mathcal K$ may be infinite-dimensional.
\end{proposition}

\begin{proof}
Let $W_0$ be the weight in Proposition~\ref{prop:noncommuting-example}. Set
\[
 w_*(z)=\frac1{|z|^2\log^2(e/|z|)}\quad(0<|z|<1/4),
 \qquad
 W_*(z)=\begin{cases}w_*(z)I,&|z|<1/4,\\W_0(z),&|z|\ge1/4.\end{cases}
\]
Set $W_*(0)=I$.
The measure $d\mueta$ is comparable with area measure on $|z|<1/4$.
With $s=\log(e/r)$, polar integration gives
\[
 \int_{|z|<1/4}w_*\,d\mueta
 \asymp_\eta\int_0^{1/4}\frac{dr}{r\log^2(e/r)}
 =\int_{\log(4e)}^\infty\frac{ds}{s^2}<\infty.
\]
For every $\delta>0$, the same substitution gives
\[
 \int_{|z|<1/4}w_*^{1+\delta}\,d\mueta
 \asymp_{\eta,\delta}\int_{\log(4e)}^\infty
           \frac{e^{2\delta s}}{s^{2+2\delta}}\,ds=\infty.
\]
The exponential term grows faster than the denominator.
Also $w_*\ge1$ on the inner disc, so $w_*^{-1/(p-1)}\le1$ there.
On that disc, the two standing integrands are
$w_*\|e\|^p$ and $w_*^{-1/(p-1)}\|e\|^{p'}$.
On its complement, the standing integrability follows from $W_0$.
Thus $W_*$ satisfies both assumptions, as well as positivity, invertibility, and measurability.

The weights $W_*$ and $W_0$ agree on $|z|>1/4$.
Proposition~\ref{prop:compact-stability} therefore proves boundedness of $P_\eta$ for $W_*$.
Choose any two distinct radii greater than $1/4$.
Their values remain noncommuting by Proposition~\ref{prop:noncommuting-example}.
The eigenvalue ratio still tends to infinity at the boundary, so uniform scalar comparison also fails.

Finally, take the Carleson square $S=\D$ and a unit vector $e\in\Hh$.
For every $\varepsilon>0$, the inner disc gives
\[
 \int_\D\|W_*^{1/p}e\|^{p+\varepsilon}\,d\mueta
 \ge\int_{|z|<1/4}w_*^{1+\varepsilon/p}\,d\mueta=\infty.
\]
The integral with exponent $p$ is finite by the standing assumption.
Thus the asserted reverse H\"older inequality already fails on $S=\D$ for every $\varepsilon>0$.
\end{proof}

This example shows directly that the averaging condition does not imply a global directional reverse H\"older estimate. Theorem~\ref{thm:operator-main} does not require that implication.

\subsection{Finite-dimensional compressions}\label{sec:finite}
We now compare a weight with its finite-dimensional compressions.
To preserve the pointwise norm, we compress $W^{2/p}$ and then take the power $p/2$.
For $p=2$ this is the usual compression of $W$. For other $p$, the two constructions need not agree.

\begin{proposition}\label{prop:finite-sections}
Let $W$ satisfy the standing assumptions. Let $\Hh_n\subset\Hh$ be increasing nonzero finite-dimensional subspaces with dense union, and let $\Pi_n$ be their orthogonal projections. Define weights on $\Hh_n$ by
\[
 U_n(z)=\left(\Pi_n W(z)^{2/p}\Pi_n\big|_{\Hh_n}\right)^{1/2},
 \qquad W_n(z)=U_n(z)^p.
\]
They satisfy the standing assumptions. For every Carleson square $S$,
\begin{align}
 \|E_S\|_{L^p(W)\to L^p(W)}
 &=\sup_n\|E_S\|_{L^p(W_n)\to L^p(W_n)},\label{eq:finite-expectation}\\
 [W]_{\mathbf B_p^{\rm av}(\eta)}
 &=\sup_n[W_n]_{\mathbf B_p^{\rm av}(\eta)}.\label{eq:finite-characteristic}
\end{align}
If this characteristic is finite, then
\begin{equation}\label{eq:finite-projection}
 \|P_\eta\|_{L^p(W)\to L^p(W)}
 =\sup_n\|P_\eta\|_{L^p(W_n)\to L^p(W_n)}.
\end{equation}
Both sides are finite by Theorem~\ref{thm:operator-main}.
\end{proposition}

\begin{proof}
Put $U=W^{1/p}$ and $X=L^p(W)$.

\smallskip\noindent\emph{1. Verify the compressed weights.}
For $e\in\Hh_n$, the definition gives
\begin{equation}\label{eq:finite-norm-preserve}
 \|U_ne\|^2=\langle U^2e,e\rangle=\|Ue\|^2.
\end{equation}
In particular, $\|U_ne\|\ge\|U^{-1}\|^{-1}\|e\|$ at almost every point.
Thus $U_n$ is positive and invertible on $\Hh_n$.
To estimate its inverse, use \eqref{eq:finite-norm-preserve} and self-adjointness of $U$:
\[
 \|U_n^{-1}e\|
 =\sup_{0\ne h\in\Hh_n}\frac{|\langle e,h\rangle|}{\|U_nh\|}
 =\sup_{0\ne h\in\Hh_n}\frac{|\langle U^{-1}e,Uh\rangle|}{\|Uh\|}
 \le\|U^{-1}e\|.
\]
These two pointwise estimates give the standing integrability for $W_n$.
The matrix entries of the compressed operator are measurable. Taking its positive square root and its $p$th power preserves measurability.
Equation~\eqref{eq:finite-norm-preserve} identifies $L^p(W_n)$ isometrically with
\[
 X_n=\{f\in X:f(z)\in\Hh_n\text{ almost everywhere}\}.
\]
These subspaces increase with $n$. They are closed in $X$.
Indeed, suppose $f_j\in X_n$ and $f_j\to f$ in $X$.
Then $Uf_j\to Uf$ in ordinary $L^p$.
A subsequence converges almost everywhere.
At each such point, multiplication by the bounded operator $U(z)^{-1}$ gives $f_j(z)\to f(z)$ along that subsequence.
Since $\Hh_n$ is closed, $f(z)\in\Hh_n$ almost everywhere.

\smallskip\noindent\emph{2. Prove density in the weighted norm.}
Fix $f\in X$. For integers $m\ge2$, put
\[
 G_m=\{z:|z|\le1-1/m,\ \|U(z)\|\le m,\
                    \|U(z)^{-1}\|\le m,\ \|f(z)\|\le m\}.
\]
The operator norms are measurable, since $\Hh$ is separable and their suprema can be taken over a countable dense subset of the unit sphere.
The sets $G_m$ increase to $\D$ up to a null set.
Hence
\[
 \|f-\mathbf1_{G_m}f\|_X^p
 =\int_{\D\setminus G_m}\|Uf\|^p\,d\mueta\longrightarrow0.
\]
For fixed $m$, put $f_{m,n}=\mathbf1_{G_m}\Pi_nf\in X_n$.
Since $\Pi_n\to I$ strongly, $U(f_{m,n}-\mathbf1_{G_m}f)$ tends to zero almost everywhere.
Moreover,
\[
 \|U(f_{m,n}-\mathbf1_{G_m}f)\|^p
 \le(2m^2)^p\mathbf1_{G_m}.
\]
This is an integrable bound independent of $n$.
Dominated convergence gives $f_{m,n}\to\mathbf1_{G_m}f$ in $X$.
First choosing $m$ and then $n$ proves that $\bigcup_nX_n$ is dense in $X$.

\smallskip\noindent\emph{3. Compare the averaging norms.}
For $f\in X_n$, its weak average on $S$ lies in $\Hh_n$.
Indeed, its pairing with every vector in $\Hh_n^\perp$ is zero.
Thus $E_S$ preserves $X_n$, and its restriction is the average on $L^p(W_n)$ under the isometry above.
For a fixed $S$, the operator $E_S$ is bounded on $X$ by Section~\ref{sec:prelim-scalar}.
Its norm is the supremum of its norms on the dense union of the $X_n$.
This proves \eqref{eq:finite-expectation}.
Raise both sides to the power $p$ and take the supremum in $S$.
The two suprema can be interchanged, which proves \eqref{eq:finite-characteristic}.

\smallskip\noindent\emph{4. Compare the projection norms.}
Assume that the characteristic in \eqref{eq:finite-characteristic} is finite.
Theorem~\ref{thm:operator-main} gives bounded projections on $X$ and on each $L^p(W_n)$.
We first identify their restrictions.

For every $f\in X$ and fixed $z\in\D$, the kernel defines a weak vector $v_z(f)\in\Hh$ by
\[
 \langle v_z(f),h\rangle
 =\int_\D K_\eta(z,\zeta)\langle f(\zeta),h\rangle\,d\mueta(\zeta)
 \qquad(h\in\Hh).
\]
Indeed, weighted H\"older's inequality gives
\[
 \left|\int_\D K_\eta(z,\zeta)\langle f(\zeta),h\rangle\,d\mueta(\zeta)\right|
 \le (1-|z|)^{-\eta-2}C_-(W)\|f\|_X\|h\|.
\]
The Riesz representation theorem therefore defines $v_z(f)$.
It also gives
\[
 \|v_z(f)-v_z(g)\|
 \le(1-|z|)^{-\eta-2}C_-(W)\|f-g\|_X.
\]
Choose $f_j\in\mathscr D_W$ with $f_j\to f$ in $X$.
For every fixed $z$, we have $v_z(f_j)\to v_z(f)$ in $\Hh$.
On the dense class, $v_z(f_j)$ agrees almost everywhere with $P_\eta f_j(z)$.
Boundedness gives $UP_\eta f_j\to UP_\eta f$ in ordinary $L^p$.
Take an almost everywhere convergent subsequence and multiply by $U(z)^{-1}$.
Comparison of the two limits shows that $P_\eta f(z)=v_z(f)$ almost everywhere.

If $f\in X_n$, the defining pairing of $v_z(f)$ vanishes on $\Hh_n^\perp$.
Thus $P_\eta$ preserves $X_n$.
On $\mathscr D_{W_n}$, its restriction agrees with the Bochner kernel integral defining the projection on $L^p(W_n)$.
Density in that space shows that these two bounded operators coincide.

Let $M$ be the supremum on the right of \eqref{eq:finite-projection}.
Restriction gives $M\le\|P_\eta\|_{X\to X}<\infty$.
Conversely, $\|P_\eta f\|_X\le M\|f\|_X$ holds on every $X_n$.
The density proved in Step 2 extends this bound to all of $X$.
Hence $\|P_\eta\|_{X\to X}\le M$, which proves \eqref{eq:finite-projection}.
\end{proof}

Theorem~\ref{thm:operator-main} gives a uniform bound for these compressions. Their projection norms therefore remain bounded as the dimension increases. The identities above also show that any better dimension-free matrix estimate passes to the Hilbert-space setting with the same bound.

The general exponent $\gamma_p=1+1/p+1/(p-1)$ comes from summing two geometric series. It is larger than the finite-dimensional exponent $\theta_p=1-1/p+1/(p-1)$. At $p=2$, almost orthogonality removes this loss. It remains to determine whether one can obtain the exponent $\theta_p$, or a smaller exponent, with a dimension-free constant for every $p$. Scalar weights require at least the scalar sharp exponent $\max\{1,1/(p-1)\}$ \cite{PottReguera13}. The present proof does not give matching lower examples for operator weights.

\appendix
\section{Finite-dimensional estimates}\label{sec:matrix-class}
\subsection{Matrix characteristics}
In this appendix $d<\infty$ and $W$ is a $d\times d$ weight. Reducing matrices and matrix $A_p$ classes are treated in \cite{Goldberg03,LauzonTreil07,Volberg97}. Convex-body sparse estimates were developed in \cite{NPTV17}. We use
\[
 \|f\|_{L^p(W)}^p=\int_\D |W(z)^{1/p}f(z)|^p\,d\mueta(z).
\]

\begin{proposition}\label{prop:scalar-reduction}
If $W=wI_d$, then
\[
 [W]_{\mathbf B_p^{\rm av}(\eta)}
 =\sup_S
 \left(\fint_Sw\,d\mueta\right)
 \left(\fint_Sw^{-1/(p-1)}\,d\mueta\right)^{p-1}
 =[w]_{B_p(\eta)}.
\]
For $p=2$ and general finite-dimensional $W$,
\[
 \|E_S\|_{L^2(W)\to L^2(W)}^2
 =\left\|
 \left(\fint_SW\,d\mueta\right)^{1/2}
 \left(\fint_SW^{-1}\,d\mueta\right)^{1/2}
 \right\|^2.
\]
\end{proposition}

\begin{proof}
Fix a Carleson square $S$ and first take $W=wI_d$.  Since $E_Sf$ is constant on $S$,
\[
 \|E_Sf\|_{L^p(w)}
 =\left(\int_Sw\,d\mueta\right)^{1/p}
   \left|\fint_S f\,d\mueta\right|.
\]
H\"older's inequality bounds the norm of the vector average map $f\mapsto\fint_S f\,d\mueta$ on $L^p(w;S)$ by
\[
 \frac1{\mueta(S)}
 \left(\int_Sw^{-1/(p-1)}\,d\mueta\right)^{1/p'}.
\]
This bound is attained by a scalar multiple of $\mathbf1_Sw^{-1/(p-1)}e$, where $e$ is any fixed unit vector.
Hence
\begin{align*}
 \|E_S\|_{L^p(w)\to L^p(w)}^p
 &=\frac1{\mueta(S)^p}
   \left(\int_Sw\,d\mueta\right)
   \left(\int_Sw^{-1/(p-1)}\,d\mueta\right)^{p-1}\\
 &=\left(\fint_Sw\,d\mueta\right)
   \left(\fint_Sw^{-1/(p-1)}\,d\mueta\right)^{p-1}.
\end{align*}
Taking the supremum gives the scalar identity.

Now let $p=2$ and let $W$ be matrix-valued.  Put
\[
 A_S=\fint_SW\,d\mueta,
 \qquad B_S=\fint_SW^{-1}\,d\mueta.
\]
For $e\in\mathbb C^d$, the constant embedding $J_Se=\mathbf1_Se$ satisfies
\[
 \|J_Se\|_{L^2(W)}^2
 =\mueta(S)\langle A_Se,e\rangle.
\]
On the other hand, for the averaging map $L_Sf=\fint_Sf\,d\mueta$ and every $e\in\mathbb C^d$,
\begin{align*}
 |\langle L_Sf,e\rangle|
 &=\left|\fint_S
   \langle W^{1/2}(z)f(z),W(z)^{-1/2}e\rangle
   \,d\mueta(z)\right|\\
 &\le \mueta(S)^{-1/2}\|f\|_{L^2(W)}
       \langle B_Se,e\rangle^{1/2}.
\end{align*}
For $f\in L^2(W)$ and $e\in\mathbb C^d$, direct computation also gives
\[
 \langle L_Sf,e\rangle
 =\left\langle f,\mueta(S)^{-1}\mathbf1_SW^{-1}e
   \right\rangle_{L^2(W)}.
\]
The vector on the right belongs to $L^2(W)$ by the standing assumptions. Hence
\[
 L_S^*e=\frac1{\mueta(S)}\mathbf1_SW^{-1}e,
 \qquad
 L_SL_S^*=\frac1{\mueta(S)}B_S.
\]
Since $\|J_Sx\|_{L^2(W)}=\mueta(S)^{1/2}|A_S^{1/2}x|$ and $E_S=J_SL_S$,
\begin{align*}
 \|E_S\|^2
 &=\mueta(S)\|A_S^{1/2}L_S\|^2\\
 &=\mueta(S)\|A_S^{1/2}L_SL_S^*A_S^{1/2}\|\\
 &=\|A_S^{1/2}B_SA_S^{1/2}\|
 =\|A_S^{1/2}B_S^{1/2}\|^2.
\end{align*}
The last equality is the identity $\|TT^*\|=\|T\|^2$ with $T=A_S^{1/2}B_S^{1/2}$. This proves the formula. The argument also holds on a separable Hilbert space. It is the local averaging computation in \cite{AlemanConstantin12}.
\end{proof}

A reducing matrix for the norm
$e\mapsto(\fint_S|W^{1/p}e|^p)^{1/p}$ is a positive matrix whose Euclidean norm is comparable to it, uniformly in $e$. Such matrices exist by finite-dimensional ellipsoid approximation; see \cite{Goldberg03,LauzonTreil07}.

\begin{proposition}\label{prop:equiv-char}
For fixed $1<p<\infty$ and $d<\infty$,
\[
 [W]_{\mathbf B_p^{\rm av}(\eta)}
 \asymp_{p,d}
 [W]_{\mathbf B_p^{\rm R}(\eta)}.
\]
Equivalently, if $\mathcal W_{p,S}$ and $\mathcal V_{p',S}$ are reducing matrices satisfying
\begin{align*}
 |\mathcal W_{p,S}e|^p&\asymp_{p,d}\fint_S|W^{1/p}(z)e|^p\,d\mueta(z),\\
 |\mathcal V_{p',S}e|^{p'}&\asymp_{p,d}\fint_S|W^{-1/p}(z)e|^{p'}\,d\mueta(z),
\end{align*}
then
\[
 [W]_{\mathbf B_p^{\rm R}(\eta)}
 \asymp_{p,d}\sup_S\|\mathcal W_{p,S}\mathcal V_{p',S}\|^p.
\]
\end{proposition}

\begin{proof}
Fix $S$. All integrals and $L^p$ norms in this paragraph use the probability measure $\mueta(S)^{-1}\mueta$ on $S$. This normalization does not change the norm of $E_S$. Write $U=W^{1/p}$, $A=\mathcal W_{p,S}$ and $B=\mathcal V_{p',S}$. Thus the norms
\[
 \rho(e)=\left(\fint_S|Ue|^p\right)^{1/p},\qquad
 \tau(e)=\left(\fint_S|U^{-1}e|^{p'}\right)^{1/p'}
\]
are comparable to $|Ae|$ and $|Be|$, respectively.
The average map $Lf=\fint_Sf$ is onto, since it maps each constant function to its value.
For $h\in\mathbb C^d$, put $F=Uf$. Then
\[
 \langle Lf,h\rangle=\fint_S\langle F,U^{-1}h\rangle,
 \qquad \|F\|_p=\|f\|_{L^p(W;S)}.
\]
Ordinary $L^p$ duality shows that the norm of this functional is $\tau(h)$.
For the quotient norm $q(e)=\inf_{Lf=e}\|f\|_{L^p(W;S)}$, it follows that
\[
 q^*(h)=\sup_{\|f\|_{L^p(W;S)}\le1}|\langle Lf,h\rangle|=\tau(h).
\]
The dual norm of $h\mapsto|Bh|$ is $e\mapsto|B^{-1}e|$, since $B$ is positive. Taking dual norms therefore gives $q(e)\asymp_{p,d}|B^{-1}e|$. The constant embedding has norm $\rho(e)$, so
\[
 \|E_S\|=\sup_{e\ne0}\frac{\rho(e)}{q(e)}
 \asymp_{p,d}\|AB\|.
\]
For the last comparison, write $e=Bv$ and take the supremum over $v\ne0$.
We next compare the double average. Fix an orthonormal basis $e_1,\ldots,e_d$. For any matrix $M$ and $1<r<\infty$,
\[
 \|M\|\asymp_{r,d}\left(\sum_{i=1}^d|Me_i|^r\right)^{1/r}.
\]
For fixed $z$, first take the adjoint inside the operator norm. Then apply the reducing property of $B$ to the fixed vectors $U(z)e_i$. This gives
\begin{align*}
 &\left(\fint_S\|U(z)U(\zeta)^{-1}\|^{p'}\,d\mueta(\zeta)\right)^{1/p'}\\
 &\quad=\left(\fint_S\|U(\zeta)^{-1}U(z)\|^{p'}\,d\mueta(\zeta)\right)^{1/p'}\\
 &\quad\asymp_{p,d}
 \left(\sum_{i=1}^d\fint_S|U(\zeta)^{-1}U(z)e_i|^{p'}\,d\mueta(\zeta)\right)^{1/p'}\\
 &\quad\asymp_{p,d}\left(\sum_{i=1}^d|BU(z)e_i|^{p'}\right)^{1/p'}
 \asymp_{p,d}\|BU(z)\|=\|U(z)B\|.
\end{align*}
The last equality also follows by taking the adjoint. It does not assert that the two matrices commute. For the average in $z$, the reducing property of $A$ gives
\[
 \fint_S\|U(z)B\|^p\,d\mueta(z)
 \asymp_{p,d}\sum_{i=1}^d\fint_S|U(z)Be_i|^p\,d\mueta(z)
 \asymp_{p,d}\sum_{i=1}^d|ABe_i|^p
 \asymp_{p,d}\|AB\|^p.
\]
Taking the supremum over $S$ proves both assertions. These are identities and norm comparisons on a single probability space. No geometry of the underlying set is used.
\end{proof}

\begin{remark}
We use the averaging characteristic in arbitrary dimension. We use the Roudenko characteristic for the finer matrix estimate. Their comparison constants may depend on dimension.
\end{remark}

\subsection{The matrix sparse estimate}\label{sec:tree-sparse}
We now prove the finer finite-dimensional bound in Theorem~\ref{thm:matrix-main}. We specialize the nonhomogeneous matrix sparse theory \cite{BenitoEtAl26} to boundary tents. Only the local matrix estimates and the nested tree structure enter.

Fix a dyadic grid $\mathcal G$ on $\partial\D$ and write
$\T(\mathcal G)=\{S(I):I\in\mathcal G\}$. Its tents are nested or disjoint. By Lemma~\ref{lem:full-tree-sparse}, the two children $Q_1,Q_2$ of $Q$ satisfy
\begin{equation}\label{eq:child-ratio}
 \mueta(Q_1)+\mueta(Q_2)\le\kappa_\eta\mueta(Q),
 \qquad \kappa_\eta<1.
\end{equation}
Thus the full tree is sparse, with disjoint top halves $G_Q=Q\setminus(Q_1\cup Q_2)$.

We use the definition of sparseness from Section~\ref{sec:separated}, with $\mu$ in place of $\mueta$ on a general measure space $(X,\mu)$.
For a subfamily $\Ssp\subset\T$ define the vector sparse form
\[
 \A_{\Ssp}(f,g)
 =\sum_{Q\in\Ssp}\mu(Q)
 \fint_Q\fint_Q |\ip{f(x)}{g(y)}|
 \,d\mu(x)d\mu(y).
\]
Here each average uses $\mu$. On the disc we take $\mu=\mueta$.

\begin{definition}\label{def:rooted-tree}
A rooted nested tree $\T$ on $(X,\mu)$ is a countable family of measurable sets of finite positive measure. Write $\T=\bigcup_{n\ge0}\T_n$ for its generations. Sets in the same generation are pairwise disjoint. Each $R\in\T_{n+1}$ has a unique parent $\widehat R\in\T_n$ with $R\subset\widehat R$. Any two sets in $\T$ are disjoint or one contains the other. Descendants of $Q$, including $Q$ itself, form $\T(Q)$. Each descendant is reached from $Q$ by finitely many child steps. Set $\mathscr S(Q)=\mathscr S\cap\T(Q)$ for a subfamily $\mathscr S$. Inclusions in this appendix allow equality. The localized maximal operator is
\[
 M_Q^{\T}h(x)=\sup_{x\in R\in\T(Q)}\fint_R|h|\,d\mu
\]
for $x\in Q$, and it is zero off $Q$.
\end{definition}

This operator has weak type $(1,1)$ with constant $1$. Indeed, the maximal descendants $R$ for which $\fint_R|h|\,d\mu>\lambda$ are pairwise disjoint. They cover the level set. Such maximal sets exist because each descendant has only finitely many ancestors below $Q$. Hence
\[
 \mu\{M_Q^{\T}h>\lambda\}
 \le\sum_R\mu(R)
 \le\lambda^{-1}\int_Q|h|\,d\mu.
\]
The children need not cover their parent.

We first prove the variable-weight Carleson embedding on a nested tree. This is the main step needed for the matrix estimate.

\begin{lemma}\label{lem:tree-variable-CE}
Let $1<s<\infty$, set $s'=s/(s-1)$, and let $\T$ be a rooted nested tree. For each $Q\in\T$, let $w_Q$ be positive almost everywhere on $Q$, with $0<\int_Qw_Q\,d\mu<\infty$.  Assume that for some $A\ge1$,
\[
 \|w_Pw_Q^{-1}\|_{L^\infty(Q)}
 \le A\,\langle w_P\rangle_Q\,\langle w_Q\rangle_Q^{-1},
 \qquad Q\subset P,
\]
where $\langle h\rangle_Q=\fint_Qh\,d\mu$.  Let $\{\beta_Q\}_{Q\in\T}$ be nonnegative and suppose that
\[
 \sum_{Q\in\T(P)}
 \langle w_P\rangle_Q\,\langle w_Q\rangle_Q^{s-1}\,\beta_Q
 \le C_2\,\mu(P)\,\langle w_P\rangle_P,
 \qquad P\in\T.
\]
Then
\[
 \sum_{Q\in\T}
 \left\langle w_Q^{1/s'}|f|\right\rangle_Q^{s}\beta_Q
 \lesssim_s A^{1+1/s'}C_2\,\|f\|_{L^s(\mu)}^s.
\]
The implicit constant is independent of the branching of the tree and of the ambient geometry.
\end{lemma}

\begin{proof}
This is the tree form of \cite[Theorem~4.3]{BenitoEtAl26}. We give a proof using only nesting.
Put $a_Q=\langle w_Q\rangle_Q$. Replace $w_Q$ by $w_Q/a_Q$ and $\beta_Q$ by $a_Q^{s-1}\beta_Q$. Both sides of the required estimate are unchanged. The hypotheses become
\[
 \langle w_Q\rangle_Q=1,\qquad
 w_P\le A\langle w_P\rangle_Qw_Q\quad(Q\subset P),
 \qquad
 \sum_{Q\subset P}\langle w_P\rangle_Q\beta_Q\le C_2\mu(P).
\]
Let $q=s'$ and $\alpha_Q=\beta_Q/\mu(Q)$. The sequence norm is
$\|x\|_{\ell^q(\beta)}=(\sum_Q|x_Q|^q\beta_Q)^{1/q}$.
Terms with $\beta_Q=0$ may be omitted. The map
$(Hf)_Q=\langle w_Q^{1/q}f\rangle_Q$ satisfies
\[
 \sum_Q(Hf)_Qx_Q\beta_Q
 =\int_X f\left(\sum_Qx_Q\alpha_Qw_Q^{1/q}\mathbf1_Q\right)d\mu.
\]
By duality it suffices to bound
\[
 H^*x=\sum_Q x_Q\alpha_Qw_Q^{1/q}\mathbf1_Q
\]
in $L^q(\mu)$ by a constant times $\|x\|_{\ell^q(\beta)}$. First restrict $\beta$ to a finite subfamily. This preserves the testing inequality. All sums below are over this subfamily, with the order inherited from $\T$. Replacing $x_Q$ by $|x_Q|$ only enlarges $|H^*x|$, so we take $x_Q\ge0$.

For any ordered nonnegative sequence $(b_j)_{j=1}^n$, the fundamental theorem of calculus gives
\begin{equation}\label{eq:chain-calculus}
 \Big(\sum_{j=1}^n b_j\Big)^q
 \le q\sum_{j=1}^n b_j\Big(\sum_{i=1}^j b_i\Big)^{q-1}.
\end{equation}
To see this, put $B_j=\sum_{i=1}^j b_i$ and $B_0=0$. Then
$B_j^q-B_{j-1}^q=q\int_{B_{j-1}}^{B_j}t^{q-1}\,dt\le qb_jB_j^{q-1}$. Sum in $j$.
At each point, order the active sets from largest to smallest. Set
\[
 k_{R,P}=\langle w_P\rangle_R^{1/q}\quad(R\subset P),
 \qquad z_R=\sum_{P\supset R}x_P\alpha_Pk_{R,P},
 \qquad I=\sum_R x_Rz_R^{q-1}\beta_R.
\]
On $R$, compatibility gives
$\sum_{P\supset R}x_P\alpha_Pw_P^{1/q}
 \le A^{1/q}w_R^{1/q}z_R$.
The term indexed by $R$ in \eqref{eq:chain-calculus} is therefore at most
\[
 qx_R\alpha_R w_R^{1/q}
       (A^{1/q}w_R^{1/q}z_R)^{q-1}\mathbf1_R
 =qA^{(q-1)/q}x_R\alpha_Rz_R^{q-1}w_R\mathbf1_R.
\]
Integrate and use $\int_Rw_R=\mu(R)$ and $\alpha_R\mu(R)=\beta_R$. We get
\begin{equation}\label{eq:Hstar-I}
 \|H^*x\|_q^q\le qA^{(q-1)/q}I.
\end{equation}
For $R\subset Q\subset P$, average $w_P\le A\langle w_P\rangle_Qw_Q$ over $R$. Also use the testing inequality under $Q$. These give
\[
 k_{R,P}\le A^{1/q}k_{R,Q}k_{Q,P},
 \qquad \sum_{R\subset Q}k_{R,Q}^q\beta_R\le C_2\mu(Q).
\]
Apply \eqref{eq:chain-calculus} to $z_R$, again ordering ancestors from largest to smallest.
For a fixed ancestor $Q$ of $R$, the partial sum ending at $Q$ is
\[
 \sum_{P\supset Q}x_P\alpha_Pk_{R,P}
 \le A^{1/q}k_{R,Q}\sum_{P\supset Q}x_P\alpha_Pk_{Q,P}
 =A^{1/q}k_{R,Q}z_Q.
\]
The term added at this step is $x_Q\alpha_Qk_{R,Q}$.
Substituting these two expressions into \eqref{eq:chain-calculus} gives
\[
 z_R^q\le qA^{(q-1)/q}
 \sum_{Q\supset R}x_Q\alpha_Qk_{R,Q}^q z_Q^{q-1}.
\]
Multiply by $\beta_R$ and sum in $R$. Interchange the two finite sums and then apply descendant testing:
\begin{align*}
 \|z\|_{\ell^q(\beta)}^q
 &\le qA^{(q-1)/q}\sum_Qx_Q\alpha_Qz_Q^{q-1}
           \sum_{R\subset Q}k_{R,Q}^q\beta_R\\
 &\le qA^{(q-1)/q}C_2\sum_Qx_Q\alpha_Qz_Q^{q-1}\mu(Q)\\
 &=qA^{(q-1)/q}C_2I\le qAC_2I.
\end{align*}
The last inequality uses $A\ge1$.
H\"older gives $I\le\|x\|_{\ell^q(\beta)}\|z\|_{\ell^q(\beta)}^{q-1}$.
If $\|z\|_{\ell^q(\beta)}>0$, divide the preceding bound by its $(q-1)$st power to get
$\|z\|_{\ell^q(\beta)}\le qAC_2\|x\|_{\ell^q(\beta)}$.
If this norm is zero, the estimate is immediate. Thus
$I\lesssim_q(AC_2)^{q-1}\|x\|_{\ell^q(\beta)}^q$.
Together with \eqref{eq:Hstar-I}, this proves
\[
 \|H^*x\|_q^q
 \lesssim_q A^{(q-1)/q}(AC_2)^{q-1}\|x\|_{\ell^q(\beta)}^q.
\]
The exponent of $A$ in this last bound is $(q-1)(1+1/q)$.
The norms of $H$ and $H^*$ agree. Since $s/q=1/(q-1)$, raising the bound for $\|H^*\|^q$ to this power gives
\[
 \|Hf\|_{\ell^s(\beta)}^s
 \lesssim_s
 \big[A^{(q-1)(1+1/q)}C_2^{q-1}\big]^{1/(q-1)}\|f\|_s^s
 =A^{1+1/q}C_2\|f\|_s^s.
\]
The constant is independent of the finite subfamily. Exhaust the countable tree by finite subfamilies and use monotone convergence. Reversing the normalization completes the proof.
\end{proof}

\begin{remark}\label{rem:tree-CE-dependence}
Lemma~\ref{lem:tree-variable-CE} is the nested-tree case of \cite[Theorem~4.3]{BenitoEtAl26}. The proof uses only the order of the sets, the compatibility inequality, and descendant testing. It gives the same power $A^{1+1/s'}$.
\end{remark}

\begin{theorem}\label{thm:tree-matrix}
Let $1<p<\infty$ and $1\le d<\infty$. Let $(X,\mu)$ be a measure space and $\T$ a rooted nested tree in the sense of Definition~\ref{def:rooted-tree}. Let $\Ssp\subset\T$ be $\delta$-sparse. Let $W$ be a measurable $d\times d$ matrix function that is Hermitian and positive definite almost everywhere. Assume that, for every $Q\in\T$ and $e\in\mathbb C^d$,
\[
 \int_Q|W^{1/p}e|^p\,d\mu<\infty,
 \qquad
 \int_Q|W^{-1/p}e|^{p'}\,d\mu<\infty.
\]
Put
\[
 V=W^{-p'/p},\qquad V^{1/p'}=W^{-1/p}.
\]
In this theorem $[W]_{A_p(\T)}$ denotes the Roudenko double-average characteristic with the supremum restricted to $Q\in\T$, namely
\[
 [W]_{A_p(\T)}
 :=\sup_{Q\in\T}\fint_Q
 \left(\fint_Q\|W(x)^{1/p}W(y)^{-1/p}\|^{p'}\,d\mu(y)\right)^{p/p'}d\mu(x).
\]
For a positive scalar weight $u$ with $0<u(P):=\int_Pu\,d\mu<\infty$ for every $P\in\T$, write
\[
 [u]_{A_\infty(\T)}
 :=\sup_{P\in\T}\frac1{u(P)}\int_P M_P^{\T}(u\mathbf1_P)\,d\mu,
\]
and define the directional Fujii--Wilson characteristics
\[
 [W]_{A^{\rm sc}_{p,\infty}(\T)}
 :=\sup_{|e|=1}\big[|W^{1/p}e|^p\big]_{A_\infty(\T)},
 \qquad
 [V]_{A^{\rm sc}_{p',\infty}(\T)}
 :=\sup_{|e|=1}\big[|V^{1/p'}e|^{p'}\big]_{A_\infty(\T)}.
\]
If these quantities and $[W]_{A_p(\T)}$ are finite, then, for $f\in L^p(\mu;\mathbb C^d)$ and $g\in L^{p'}(\mu;\mathbb C^d)$,
\begin{equation}\label{eq:tree-matrix-mixed}
 \A_{\Ssp}(W^{-1/p}f,W^{1/p}g)
 \lesssim_{p,d,\delta}
 [W]_{A_p(\T)}^{1/p}
 [W]_{A^{\rm sc}_{p,\infty}(\T)}^{1/p'}
 [V]_{A^{\rm sc}_{p',\infty}(\T)}^{1/p}
 \|f\|_{L^p(\mu)}\|g\|_{L^{p'}(\mu)}.
\end{equation}
Moreover, suppose that $\T=\T(\mathcal G)$ is one of the dyadic Carleson tent trees in $(\D,\mueta)$ and that $W$ has finite full Carleson-square Roudenko characteristic
\[
 B:=[W]_{\mathbf B_p^{\rm R}(\eta)}<\infty.
\]
Then
\begin{equation}\label{eq:tree-matrix-bound}
 \A_{\Ssp}(W^{-1/p}f,W^{1/p}g)
 \lesssim_{p,d,\delta,\eta}
 B^{\theta_p}\|f\|_{L^p(\mueta)}\|g\|_{L^{p'}(\mueta)},
\end{equation}
where
\[
 \theta_p=1+\frac1{p-1}-\frac1p.
\]
\end{theorem}

\begin{proof}
We first prove the mixed estimate on an abstract nested tree. We then bound its three characteristics by the single Bergman characteristic $B$. The first part uses only nesting and scalar testing. Neither part needs the children to cover their parent. No reverse H\"older inequality is used. We may first take $\Ssp$ finite. All estimates below are uniform in this choice.

\smallskip
\noindent\emph{Reducing matrices.}
For every $Q\in\T$, finite-dimensional norm geometry gives positive reducing matrices $\mathcal W_Q$ and $\mathcal V_Q$ such that
\begin{align}
 |\mathcal W_Qe|^p
 &\asymp_{p,d}\fint_Q|W(x)^{1/p}e|^p\,d\mu(x),\notag\\
 |\mathcal V_Qe|^{p'}
 &\asymp_{p,d}\fint_Q|W(x)^{-1/p}e|^{p'}\,d\mu(x),\notag
\end{align}
and
\begin{equation}\label{eq:redproducttree}
 \|\mathcal W_Q\mathcal V_Q\|
 \lesssim_{p,d}[W]_{A_p(\T)}^{1/p}.
\end{equation}
These are one-set statements on the probability space $(Q,\mu(Q)^{-1}\mu)$; no geometric property of $Q$ is involved.

\smallskip
\noindent\emph{The local scalar embeddings.}
For fixed $Q$, set
\[
 a(x)=\mathcal V_Q^{-1}W(x)^{-1/p}f(x),
 \qquad b(y)=\mathcal W_Q^{-1}W(y)^{1/p}g(y).
\]
Since the reducing matrices are positive,
\[
 |\langle\mathcal V_Qa(x),\mathcal W_Qb(y)\rangle|
 =|\langle\mathcal W_Q\mathcal V_Qa(x),b(y)\rangle|
 \le\|\mathcal W_Q\mathcal V_Q\|\,|a(x)|\,|b(y)|.
\]
Average in both variables and use \eqref{eq:redproducttree}. We obtain
\begin{align*}
 &\fint_Q\fint_Q
 |\langle W^{-1/p}(x)f(x),W^{1/p}(y)g(y)\rangle|
 \,d\mu(x)d\mu(y)\\
 &\quad\lesssim_{p,d}[W]_{A_p(\T)}^{1/p}
 \left\langle|\mathcal V_Q^{-1}W^{-1/p}f|\right\rangle_Q
 \left\langle|\mathcal W_Q^{-1}W^{1/p}g|\right\rangle_Q.
\end{align*}
For the second factor set
\[
 S_{p'}^Wg
 =\left(\sum_{Q\in\Ssp}
 \left\langle\|\mathcal W_Q^{-1}W^{1/p}\|\,|g|\right\rangle_Q^{p'}
 \mathbf1_Q\right)^{1/p'}.
\]
Apply Lemma~\ref{lem:tree-variable-CE} with $s=p'$ and
\[
 w_Q(x)=\|\mathcal W_Q^{-1}W(x)^{1/p}\|^p\mathbf1_Q(x).
\]
The power $p$ is forced by $s'=p$, since on $Q$
$w_Q^{1/s'}=\|\mathcal W_Q^{-1}W^{1/p}\|$; this agrees with the choice in \cite[Lemma~4.1]{BenitoEtAl26}.

The reducing property implies
\begin{equation}\label{eq:wQ-normalization}
 \langle w_Q\rangle_Q\asymp_{p,d}1.
\end{equation}
Indeed, taking the adjoint and then comparing column norms gives
\begin{align*}
 \langle w_Q\rangle_Q
 &=\fint_Q\|W^{1/p}\mathcal W_Q^{-1}\|^p\,d\mu\\
 &\asymp_{p,d}\sum_{i=1}^d\fint_Q|W^{1/p}\mathcal W_Q^{-1}e_i|^p\,d\mu
 \asymp_{p,d}\sum_{i=1}^d|e_i|^p=d.
\end{align*}
Here $d$ is fixed, so the last expression is comparable to $1$ with constants depending on $p,d$.
If $Q\subset P$, then on $Q$
\begin{align*}
 w_P
 &\le \|\mathcal W_P^{-1}\mathcal W_Q\|^p w_Q
   =\|\mathcal W_Q\mathcal W_P^{-1}\|^p w_Q\\
 &\lesssim_{p,d}
 \left(\fint_Q\|W(x)^{1/p}\mathcal W_P^{-1}\|^p\,d\mu(x)\right)w_Q
 \lesssim_{p,d}\langle w_P\rangle_Qw_Q.
\end{align*}
In the second line, the reducing property is applied to each fixed vector $\mathcal W_P^{-1}e_i$. In detail,
\[
 \|\mathcal W_Q\mathcal W_P^{-1}\|^p
 \lesssim_{p,d}\sum_{i=1}^d
       \fint_Q|W^{1/p}\mathcal W_P^{-1}e_i|^p\,d\mu
 \lesssim_{p,d}\fint_Q\|W^{1/p}\mathcal W_P^{-1}\|^p\,d\mu.
\]
The equality of the two product norms follows by taking adjoints.
No change in the order of the matrices is assumed.
Together with \eqref{eq:wQ-normalization}, this is the compatibility condition of Lemma~\ref{lem:tree-variable-CE}, with a constant depending only on $p,d$.

It remains to verify testing.  Put $\beta_Q=\mu(Q)$ for $Q\in\Ssp$ and $\beta_Q=0$ otherwise, and choose pairwise disjoint sets $G_Q\subset Q$ with $\mu(G_Q)\ge\delta\mu(Q)$.  For a fixed $P$, choose an orthonormal basis $e_1,\dots,e_d$ and put $h_j=\mathcal W_P^{-1}e_j$.  Finite-dimensional norm equivalence and positivity give, on $P$,
\[
 w_P\asymp_{p,d}\sum_{j=1}^d|W^{1/p}h_j|^p.
\]
The directional $A_\infty$ characteristic is unchanged by multiplying a scalar weight by a positive constant.  Since the local maximal operator is sublinear,
\begin{align*}
 \int_P M_P^{\T}(w_P\mathbf1_P)\,d\mu
 &\lesssim_{p,d}\sum_{j=1}^d
 \int_P M_P^{\T}(|W^{1/p}h_j|^p\mathbf1_P)\,d\mu\\
 &\le [W]_{A^{\rm sc}_{p,\infty}(\T)}
       \sum_{j=1}^d\int_P|W^{1/p}h_j|^p\,d\mu\\
 &\lesssim_{p,d}[W]_{A^{\rm sc}_{p,\infty}(\T)}
       \mu(P)\langle w_P\rangle_P.
\end{align*}
Therefore
\begin{align*}
 \sum_{Q\in\Ssp(P)}\langle w_P\rangle_Q\mu(Q)
 &\le\delta^{-1}\sum_{Q\in\Ssp(P)}
      \int_{G_Q}M_P^{\T}(w_P\mathbf1_P)\,d\mu\\
 &\lesssim_{p,d,\delta}[W]_{A^{\rm sc}_{p,\infty}(\T)}
      \mu(P)\langle w_P\rangle_P.
\end{align*}
Since \eqref{eq:wQ-normalization} makes the extra factor
$\langle w_Q\rangle_Q^{p'-1}$ harmless, Lemma~\ref{lem:tree-variable-CE} yields
\begin{equation}\label{eq:SpW-bound}
 \|S_{p'}^Wg\|_{p'}
 \lesssim_{p,d,\delta}
 [W]_{A^{\rm sc}_{p,\infty}(\T)}^{1/p'}\|g\|_{p'}.
\end{equation}

For the first factor put $V=W^{-p'/p}$, so $V^{1/p'}=W^{-1/p}$, and apply the same argument with $s=p$ and
\[
 v_Q(x)=\|\mathcal V_Q^{-1}V(x)^{1/p'}\|^{p'}\mathbf1_Q(x).
\]
Define
\[
 S_p^Vf
 =\left(\sum_{Q\in\Ssp}
 \left\langle\|\mathcal V_Q^{-1}V^{1/p'}\|\,|f|\right\rangle_Q^p
 \mathbf1_Q\right)^{1/p}.
\]
Here $s'=p'$, so $v_Q^{1/s'}=\|\mathcal V_Q^{-1}W^{-1/p}\|$ on $Q$.
The same column calculation, with $W,p,\mathcal W_Q$ replaced by $V,p',\mathcal V_Q$, gives
\[
 \langle v_Q\rangle_Q\asymp_{p,d}1,
 \qquad v_P\lesssim_{p,d}\langle v_P\rangle_Qv_Q
 \quad\hbox{on }Q\subset P.
\]
Use the disjoint sets $G_Q$ and the directional characteristic of $V$, as above. Then
\[
 \sum_{Q\in\Ssp(P)}\langle v_P\rangle_Q
       \langle v_Q\rangle_Q^{p-1}\mu(Q)
 \lesssim_{p,d,\delta}[V]_{A^{\rm sc}_{p',\infty}(\T)}
       \mu(P)\langle v_P\rangle_P.
\]
These are exactly the two hypotheses of Lemma~\ref{lem:tree-variable-CE} for $s=p$.
The lemma therefore gives
\begin{equation}\label{eq:SpV-bound}
 \|S_p^Vf\|_{p}
 \lesssim_{p,d,\delta}
 [V]_{A^{\rm sc}_{p',\infty}(\T)}^{1/p}\|f\|_{p}.
\end{equation}
Apply H\"older's inequality to the sum over $Q$, with measure $\mu(Q)$ on the index set. The two resulting factors are exactly $\|S_p^Vf\|_p$ and $\|S_{p'}^Wg\|_{p'}$. Consequently
\[
 \A_{\Ssp}(W^{-1/p}f,W^{1/p}g)
 \lesssim_{p,d}[W]_{A_p(\T)}^{1/p}
 \|S_p^Vf\|_p\|S_{p'}^Wg\|_{p'},
\]
and \eqref{eq:tree-matrix-mixed} follows from
\eqref{eq:SpW-bound}--\eqref{eq:SpV-bound}. Monotone convergence extends it to countable $\Ssp$. This proves the mixed estimate of \cite[Theorem~B]{BenitoEtAl26} for the present trees.

\smallskip
\noindent\emph{The Bergman characteristic.}
Assume now that $\T=\T(\mathcal G)$ is a dyadic Bergman tent tree and
$B=[W]_{\mathbf B_p^{\rm R}(\eta)}<\infty$.  For a unit vector $e\in\mathbb C^d$ set
\[
 u_e(z)=|W(z)^{1/p}e|^p.
\]
For every Carleson square $S$,
\begin{align*}
 &\left(\fint_Su_e\,d\mueta\right)
 \left(\fint_Su_e^{-1/(p-1)}\,d\mueta\right)^{p-1}\\
 &=\fint_S
 \left(\fint_S
 \left(\frac{|W(z)^{1/p}e|}{|W(\zeta)^{1/p}e|}\right)^{p'}
 d\mueta(\zeta)\right)^{p/p'}d\mueta(z)\\
 &\le\fint_S
 \left(\fint_S
 \|W(z)^{1/p}W(\zeta)^{-1/p}\|^{p'}
 d\mueta(\zeta)\right)^{p/p'}d\mueta(z)
 \le B.
\end{align*}
Thus
\[
 [u_e]_{B_p(\eta)}\le B
\]
uniformly in $e$.

We need a scalar maximal estimate on the same tree. Fix a tent $P$, let $1<q<\infty$, and suppose $v\in B_q(\eta)$. Put $\tau=v^{-1/(q-1)}$ and $K=[v]_{B_q(\eta)}$. For $a=v$ or $a=\tau$, define
\[
 M_{P,a}h(x)=\sup_{x\in R\in\T(P)}
       \frac1{a(R)}\int_R|h|a\,d\mueta.
\]
Maximal tents in a level set are disjoint, also for the measure $a\,d\mueta$. The weak $(1,1)$ argument after Definition~\ref{def:rooted-tree} applies to this measure. Split $h$ into the parts where $|h|\le\lambda/2$ and $|h|>\lambda/2$. The first part has maximal function at most $\lambda/2$. Hence
\[
 a\{M_{P,a}h>\lambda\}
 \le\frac2\lambda\int_{P\cap\{|h|>\lambda/2\}}|h|a\,d\mueta.
\]
Multiply by $r\lambda^{r-1}$ and integrate in $\lambda$. Tonelli's theorem yields, for $1<r<\infty$,
\[
 \|M_{P,a}h\|_{L^r(a;P)}^r
 \le\frac{2^rr}{r-1}\|h\|_{L^r(a;P)}^r.
\]
In particular the bound is independent of $a$ and $P$.
For $f\ge0$, put $G=[M_{P,\tau}(f/\tau)]^{q-1}$. For every $R\in\T(P)$,
\[
 \left(\fint_Rf\,d\mueta\right)^{q-1}
 \le\left(\frac{\tau(R)}{\mueta(R)}\right)^{q-1}
        \mathop{\rm ess\,inf}_{R}G
 \le\frac{K}{v(R)}\int_RG\,d\mueta.
\]
For the second inequality we used
$v(R)\tau(R)^{q-1}\le K\mueta(R)^q$ and
$\mathop{\rm ess\,inf}_R G\le\fint_R G\,d\mueta$.
Hence $(M_P^{\T}f)^{q-1}\le K M_{P,v}(G/v)$. Taking the $L^{q'}(v)$ norm, where $q'=q/(q-1)$, gives
\[
 \|M_P^{\T}f\|_{L^q(v)}^{q-1}
 \le C_qK\|G/v\|_{L^{q'}(v)}
 =C_qK\|M_{P,\tau}(f/\tau)\|_{L^q(\tau)}^{q-1}
 \le C_qK\|f\|_{L^q(v)}^{q-1}.
\]
The last step uses $\tau^{1-q}=v$, so that
$\|f/\tau\|_{L^q(\tau)}=\|f\|_{L^q(v)}$.
Thus
\begin{equation}\label{eq:scalar-Bp-maximal}
 \|M_P^{\T}\|_{L^q(v;P)\to L^q(v;P)}
 \lesssim_q[v]_{B_q(\eta)}^{1/(q-1)}.
\end{equation}
One may first truncate the tree and then use monotone convergence. This is the tree version of the estimate in \cite[Eq.~(4.7)]{PottReguera13}.

Put $\sigma_e=u_e^{-1/(p-1)}$. Since
$[\sigma_e]_{B_{p'}(\eta)}=[u_e]_{B_p(\eta)}^{1/(p-1)}$, \eqref{eq:scalar-Bp-maximal} and H\"older's inequality give
\begin{align*}
 \int_P M_P^{\T}(u_e\mathbf1_P)\,d\mueta
 &\le\|M_P^{\T}(u_e\mathbf1_P)\|_{L^{p'}(\sigma_e;P)}
       \|\mathbf1_P\|_{L^p(u_e)}\\
 &\lesssim_p[u_e]_{B_p(\eta)}u_e(P).
\end{align*}
For the last inequality, note that $u_e^{p'}\sigma_e=u_e$. Hence
$\|u_e\mathbf1_P\|_{L^{p'}(\sigma_e)}=u_e(P)^{1/p'}$,
while $\|\mathbf1_P\|_{L^p(u_e)}=u_e(P)^{1/p}$.
Also $1/(p'-1)=p-1$, which gives the first power of $[u_e]_{B_p(\eta)}$.
Therefore
\[
 [W]_{A^{\rm sc}_{p,\infty}(\T)}\lesssim_{p,\eta}B.
\]

Apply Proposition~\ref{prop:equiv-char} to $W$ and $V=W^{-p'/p}$. The averaging operators on the two dual spaces have the same norm, so
\[
 [V]_{\mathbf B_{p'}^{\rm R}(\eta)}^{1/p'}
 \asymp_{p,d}B^{1/p}.
\]
Repeating the preceding directional scalar argument for $V$ yields
\[
 [V]_{A^{\rm sc}_{p',\infty}(\T)}
 \lesssim_{p,d,\eta}B^{p'/p}.
\]
Finally $[W]_{A_p(\T)}\le B$.  Inserting these three estimates into
\eqref{eq:tree-matrix-mixed} gives the exponent
\[
 \frac1p+\frac1{p'}+\frac1p\frac{p'}p
 =1+\frac1{p(p-1)}
 =1+\frac1{p-1}-\frac1p=\theta_p.
\]
This proves \eqref{eq:tree-matrix-bound}.
\end{proof}

\begin{remark}
The sparse theorem in \cite{BenitoEtAl26} is written for real matrix weights.  The proof above uses only finite-dimensional norm geometry, positive reducing operators and scalar absolute values, so it carries over verbatim to Hermitian positive matrices.  Alternatively one may realify: identify $\mathbb C^d$ with $\mathbb R^{2d}$ and replace $W$ by its real block representation.  For the bilinear form, write
\[
 |\langle u,v\rangle_{\mathbb C^d}|
 \le |\operatorname{Re}\langle u,v\rangle|
     +|\operatorname{Im}\langle u,v\rangle|
 =|u_{\mathbb R}\cdot v_{\mathbb R}|
  +|u_{\mathbb R}\cdot (iv)_{\mathbb R}|.
\]
Thus two applications of the real estimate suffice, changing only the dimension-dependent constant.
\end{remark}

\begin{remark}\label{rem:tree-rigor}
The weak $(1,1)$ bound follows from nesting, as shown after Definition~\ref{def:rooted-tree}. The proof also uses weights and reducing matrices for every set of the full descendant tree. The matrix characteristic is therefore imposed on that tree. In particular, the testing condition in Lemma~\ref{lem:tree-variable-CE} is checked for every $P\in\T$.
\end{remark}

\subsection{The finite-dimensional bound}\label{sec:convex-bergman}
We now connect the analytic kernel with the tree sparse form.

\begin{lemma}\label{lem:dyadic-kernel}
There exist finitely many adjacent dyadic grids of boundary arcs $\mathcal G_1,\dots,\mathcal G_N$, depending only on the disc geometry, such that
\begin{equation}\label{eq:kernel-dyadic}
 \frac1{|1-\overline\zeta z|^{\eta+2}}
 \lesssim_\eta
 \sum_{j=1}^N\sum_{Q\in\T(\mathcal G_j)}
 \frac{\mathbf1_Q(z)\mathbf1_Q(\zeta)}{\mueta(Q)}.
\end{equation}
\end{lemma}

\begin{proof}
Use the three rotated grids from Section~\ref{sec:separated}. Put $\alpha=\eta+2$ and $t=|1-\overline\zeta z|$. If $t\ge1/32$, then $t^{-\alpha}\le32^\alpha$. The root $\D$ has measure $1$, so its term suffices.

Suppose $t<1/32$. Write $z=re^{2\pi ix}$ and $\zeta=\rho e^{2\pi iy}$. The two depths satisfy
\[
 1-r\le t,\qquad 1-\rho\le t,
\]
so $r,\rho>31/32$. Let $\ell$ be the shortest distance between $x$ and $y$ in $\mathbb R/\mathbb Z$. Then $0\le\ell\le1/2$ and
\[
 t^2=(1-r\rho)^2+r\rho|1-e^{2\pi i(x-y)}|^2,
 \qquad |1-e^{2\pi i(x-y)}|=2\sin(\pi\ell)\ge4\ell.
\]
It follows that $\ell\le t/(4\sqrt{r\rho})<t$. Choose an arc $J$ of length $4t$ with $x,y$ in its interior. Choose a dyadic length $H$ such that
\[
 12t<H\le24t.
\]
Such a length exists and is at most $1/2$. At this scale, the combined endpoints of the three grids are spaced by $H/3$. Since $|J|<H/3$, at least one grid has an arc $I$ of length $H$ containing $J$. The depth estimates give $z,\zeta\in Q:=S(I)$.

The tent measure formula in Lemma~\ref{lem:full-tree-sparse} gives
\[
 \mueta(Q)=H^\alpha(2-H)^{\alpha-1}
 \le2^{\alpha-1}24^\alpha t^\alpha.
\]
Thus $t^{-\alpha}\lesssim_\eta\mueta(Q)^{-1}$. This single term proves \eqref{eq:kernel-dyadic}. This is the dyadic kernel comparison used in scalar B\'ekoll\'e estimates; compare \cite{PottReguera13}.
\end{proof}

The following bilinear estimate is the analytic estimate needed for the finite-dimensional theorem. Here $\langle P_\eta f,g\rangle$ denotes the integral of the pointwise Hilbert-space pairing.

\begin{proposition}\label{prop:convex-bergman}
Let $\mathcal K$ be a Hilbert space, finite- or infinite-dimensional.  For compactly supported Bochner integrable $\mathcal K$-valued functions $f,g$,
\begin{equation}\label{eq:vector-dual-sparse}
 |\langle P_\eta f,g\rangle|
 \lesssim_\eta\sum_{j=1}^N
 \A_{\T(\mathcal G_j)}(f,g),
\end{equation}
where
\[
 \A_{\T}(f,g)
 =\sum_{Q\in\T}\mueta(Q)
   \fint_Q\fint_Q |\ip{f(\zeta)}{g(z)}|
   \,d\mueta(\zeta)d\mueta(z).
\]
The constant is independent of $\dim\mathcal K$.
\end{proposition}

\begin{proof}
Choose $r<1$ with $\operatorname{supp}f\subset\{|\zeta|\le r\}$. For every $z\in\D$ and $\zeta$ in this support,
$|1-\overline\zeta z|\ge1-r$. The absolute double integral below is therefore bounded by
$(1-r)^{-\eta-2}\|f\|_1\|g\|_1$. Fubini's theorem applies and gives
\begin{align*}
 |\langle P_\eta f,g\rangle|
 &\le \int_\D\int_\D
 \frac{|\ip{f(\zeta)}{g(z)}|}
 {|1-\overline\zeta z|^{\eta+2}}
 \,d\mueta(\zeta)d\mueta(z).
\end{align*}
Apply Lemma~\ref{lem:dyadic-kernel} pointwise to the scalar kernel.  Since all summands are nonnegative, Tonelli's theorem gives
\begin{align*}
 |\langle P_\eta f,g\rangle|
 &\lesssim_\eta
 \sum_{j=1}^N\sum_{Q\in\T(\mathcal G_j)}
 \frac1{\mueta(Q)}
 \int_Q\int_Q |\ip{f(\zeta)}{g(z)}|
 \,d\mueta(\zeta)d\mueta(z)\\
 &=\sum_{j=1}^N
 \A_{\T(\mathcal G_j)}(f,g).
\end{align*}
No finite-dimensional convexity is used.  In particular, \eqref{eq:vector-dual-sparse} is valid for arbitrary Hilbert-space-valued functions.  Finite dimensionality enters only in the weighted estimate for the sparse form in Theorem~\ref{thm:tree-matrix}.
\end{proof}

\begin{remark}\label{rem:hilbert-sparse}
Proposition~\ref{prop:convex-bergman} gives an absolute vector sparse form. Its weighted estimate here uses finite-dimensional reducing matrices. Theorem~\ref{thm:operator-main} instead uses the separated scalar kernel and Lemma~\ref{lem:layer-tree}, whose bound has no dimension parameter.
\end{remark}

We now prove the quantitative estimate in Theorem~\ref{thm:matrix-main}.

\begin{proof}[\textbf{Proof of Theorem~\ref{thm:matrix-main}}]

The implication (i)$\Rightarrow$(ii), together with \eqref{eq:main-lower}, is Proposition~\ref{prop:matrix-necessity}.  The equivalence (ii)$\Leftrightarrow$(iii) is Proposition~\ref{prop:equiv-char}.

It remains to prove (iii)$\Rightarrow$(i).  Let $F,G$ be bounded, compactly supported simple $\mathbb C^d$-valued functions. The standing integrability assumptions make $W^{-1/p}F$ and $W^{1/p}G$ Bochner integrable. These test classes are dense in unweighted $L^p$ and $L^{p'}$, respectively.  By Proposition~\ref{prop:convex-bergman},
\begin{align*}
 &|\langle W^{1/p}P_\eta(W^{-1/p}F),G\rangle|\\
 &\qquad=|\langle P_\eta(W^{-1/p}F),W^{1/p}G\rangle|\\
 &\qquad\lesssim_\eta
 \sum_{j=1}^N
 \A_{\T(\mathcal G_j)}(W^{-1/p}F,W^{1/p}G).
\end{align*}
Each full dyadic Carleson tree $\T(\mathcal G_j)$ is sparse by \eqref{eq:child-ratio}.  Its tree matrix characteristic is bounded by the full Carleson characteristic $[W]_{\mathbf B_p^{\rm R}(\eta)}$.  Theorem~\ref{thm:tree-matrix} therefore yields
\[
 |\langle W^{1/p}P_\eta(W^{-1/p}F),G\rangle|
 \lesssim_{p,d,\eta}
 [W]_{\mathbf B_p^{\rm R}(\eta)}^{\theta_p}
 \|F\|_p\|G\|_{p'}.
\]
We justify the final duality step before taking the supremum. Put $U=W^{1/p}$ and $h_F=UP_\eta(U^{-1}F)$. Finite dimensionality and \eqref{eq:standing} give
\[
 \int_\D\|U(z)\|^p\,d\mueta(z)
 \lesssim_{p,d}\sum_{i=1}^d\int_\D|U(z)e_i|^p\,d\mueta(z)<\infty.
\]
If $\operatorname{supp}F\subset\{|z|\le r\}$ with $r<1$, then
\[
 |P_\eta(U^{-1}F)(z)|
 \le(1-r)^{-\eta-2}\|U^{-1}F\|_1
 \qquad(z\in\D).
\]
Hence $h_F\in L^p$. This observation gives membership only; the uniform estimate comes from the sparse bound above.

Now take the supremum over bounded compactly supported simple $G$ with $\|G\|_{p'}\le1$. This class is dense in $L^{p'}$, so its supremum is $\|h_F\|_p$. We obtain
\[
 \|h_F\|_p
 \lesssim_{p,d,\eta}[W]_{\mathbf B_p^{\rm R}(\eta)}^{\theta_p}\|F\|_p.
\]
The bounded map $F\mapsto h_F$ extends uniquely from the dense class of compactly supported simple functions to all of $L^p$. Thus
\[
 M_{W^{1/p}}P_\eta M_{W^{-1/p}}
 :L^p(d\mueta;\mathbb C^d)\to L^p(d\mueta;\mathbb C^d)
\]
denotes a bounded extension with the stated norm. Multiplication by $W^{1/p}$ is an isometry from $L^p(W)$ onto $L^p$. Its inverse is multiplication by $W^{-1/p}$. These isometries transfer the extension to $L^p(W)$ and give \eqref{eq:main-upper}.
\end{proof}

\begin{remark}\label{rem:p2}
For $p=2$, the averaging characteristic is the square of the norm-type product in \eqref{eq:B2op-intro}. Theorem~\ref{thm:matrix-main} gives the power $3/2$ with a dimension-dependent constant. Theorem~\ref{thm:hilbert-bound} removes that dependence by a separate argument. The sparse estimate underlying this step is due to Limani--Pott; the present application uses the separated scalar kernel.
\end{remark}

\begin{remark}
The power $\theta_p$ is inherited from the general matrix sparse-form bound \cite{BenitoEtAl26}. The present proof gives no matching lower bound for the Bergman projection. Thus no sharpness statement for \eqref{eq:main-upper} is made.
\end{remark}

\section*{Statements and declarations}
\noindent\textbf{Competing interests.}
The authors declare that they have no competing interests.

\smallskip
\noindent\textbf{Data availability.}
No data sets were generated or analysed in this study.


\begingroup
\small
\begin{thebibliography}{99}
\raggedright

\bibitem{AlemanConstantin12}
A. Aleman and O. Constantin,
\emph{The Bergman projection on vector-valued $L^2$-spaces with operator-valued weights},
J. Funct. Anal. \textbf{262} (2012), 2359--2378.

\bibitem{AlemanPottReguera19}
A. Aleman, S. Pott, and M. C. Reguera,
\emph{Characterizations of a limiting class $B_\infty$ of B\'ekoll\'e--Bonami weights},
Rev. Mat. Iberoam. \textbf{35} (2019), no.~6, 1677--1692.

\bibitem{Bekolle82}
D. B\'ekoll\'e,
\emph{In\'egalit\'e \`a poids pour le projecteur de Bergman dans la boule unit\'e de $\mathbb C^n$},
Studia Math. \textbf{71} (1981/82), 305--323.

\bibitem{BB78}
D. B\'ekoll\'e and A. Bonami,
\emph{In\'egalit\'es \`a poids pour le noyau de Bergman},
C. R. Acad. Sci. Paris S\'er. A--B \textbf{286} (1978), A775--A778.

\bibitem{BenitoEtAl26}
F. Benito-de la Cigo\~na, T. Borges, F. D'Emilio, M. Pasquariello, and N. A. Wagner,
\emph{Matrix weighted $L^p$ estimates in the nonhomogeneous setting},
Adv. Math. \textbf{492} (2026), 110911.

\bibitem{ChenWangFock25}
J. Chen and M. Wang,
\emph{Fock projections on vector-valued $L^p$-spaces with matrix weights},
arXiv:2408.13537v2, 2025.

\bibitem{Goldberg03}
M. Goldberg,
\emph{Matrix $A_p$ weights via maximal functions},
Pacific J. Math. \textbf{211} (2003), no.~2, 201--220.

\bibitem{HLYY26}
T. P. Hyt\"onen, Y. Li, D. Yang, and W. Yuan,
\emph{Real-variable theory of function spaces with operator-valued $A_p$ weights in Banach spaces},
arXiv:2604.18303, 2026.

\bibitem{HuoWagnerWick21}
Z. Huo, N. A. Wagner, and B. D. Wick,
\emph{B\'ekoll\'e--Bonami estimates on some pseudoconvex domains},
Bull. Sci. Math. \textbf{170} (2021), 102993.

\bibitem{HuoWick24}
Z. Huo and B. D. Wick,
\emph{Weighted estimates of the Bergman projection with matrix weights},
in: A. A. Condori, E. Pozzi, W. T. Ross, and A. A. Sola (eds.),
\emph{Recent Progress in Function Theory and Operator Theory},
Contemp. Math. \textbf{799}, Amer. Math. Soc., 2024, 53--73.

\bibitem{LauzonTreil07}
M. Lauzon and S. Treil,
\emph{Scalar and vector Muckenhoupt weights},
Indiana Univ. Math. J. \textbf{56} (2007), no.~4, 1989--2015.

\bibitem{LimaniPott21}
A. Limani and S. Pott,
\emph{Sparse Lerner operators in infinite dimensions},
arXiv:2103.17005v4, 2022.

\bibitem{NPTV17}
F. Nazarov, S. Petermichl, S. Treil, and A. Volberg,
\emph{Convex body domination and weighted estimates with matrix weights},
Adv. Math. \textbf{318} (2017), 279--306.

\bibitem{PottReguera13}
S. Pott and M. C. Reguera,
\emph{Sharp B\'ekoll\'e estimates for the Bergman projection},
J. Funct. Anal. \textbf{265} (2013), 3233--3244.

\bibitem{Tian26}
C. Tian,
\emph{Weighted Bergman projections induced by kernel with integral representation in operator-valued setting},
Integral Equations Operator Theory \textbf{98} (2026), no.~2, Paper No.~11.

\bibitem{Volberg97}
A. Volberg,
\emph{Matrix $A_p$ weights via $S$-functions},
J. Amer. Math. Soc. \textbf{10} (1997), no.~2, 445--466.

\end{thebibliography}
\endgroup
\end{document}